\RequirePackage{fix-cm}
\documentclass[smallextended]{svjour3}       % onecolumn (second format)
\smartqed  % flush right qed marks, e.g. at end of proof
\usepackage{setspace}
\usepackage{graphicx}
\usepackage{color}
\usepackage{placeins}
\usepackage{longtable}

\usepackage{float}
\floatstyle{ruled}
\newfloat{model}{thp}{lop}
\floatname{model}{Model}
\usepackage[bookmarks=false]{hyperref}
\hypersetup{
	colorlinks=true, 
	breaklinks=true,
	urlcolor= blue, 
	linkcolor= blue, 
	citecolor=red, 
	pdftitle={},
	pdfauthor={},
}
\usepackage{amsfonts,amssymb,amsmath,textcomp,mathabx,empheq}
\usepackage{subfigure}
\usepackage{bm}
\def\rit{\mathbb{R}}

\def\nit{\mathbb{N}}

\makeatletter
\newenvironment{sqcases}{%
	\matrix@check\sqcases\env@sqcases
}{%
\endarray\right.%
}
\def\env@sqcases{%
	\let\@ifnextchar\new@ifnextchar
	\left\lbrack
	\def\arraystretch{1.2}%
	\array{@{}l@{\quad}l@{}}%
}
\makeatother

\begin{document}
\title{Invex Optimization Revisited%\thanks{Grants or other notes
}


\author{Ksenia Bestuzheva   \and
        Hassan Hijazi %etc.
}


\institute{K. Bestuzheva  \and H. Hijazi \at
              The Australian National University \\
              Data61, CSIRO \\
              ACTON, 2601, Canberra, Australia\\
              \email{u5647146@anu.edu.au}           %  \\
           \and
           H. Hijazi \at
           Los Alamos National Laboratory,\\
           Los Alamos, NM 87544, USA\\
              \email{hassan.hijazi@anu.edu.au}           %  \\
}

\date{Received: date / Accepted: date}

\maketitle
	\begin{abstract}
	Given a non-convex optimization problem, we study conditions under which every Karush-Kuhn-Tucker (KKT) point is a global optimizer. This property is known as KT-invexity and allows to identify the subset of problems where an interior point method always converges to a global optimizer. In this work, we provide necessary conditions for KT-invexity in n-dimensions and show that these conditions become sufficient in the two-dimensional case. As an application of our results, we study the Optimal Power Flow problem, showing that under mild assumptions on the variable's bounds, our new necessary and sufficient conditions are met for problems with two degrees of freedom.
	\end{abstract}

	\section*{Notations}
	\begin{tabular}{ll}
	$\partial S$ & boundary of a set $S$.\\	
	$x_i$ & $i$th component of vector $\vec x$.\\	
	$f'_{x_i} = \frac{\partial f}{\partial x_i}$ & partial derivative of $f$ with respect to $x_i$.\\	
	$||\vec x||$ & Euclidean norm of vector $\vec x$.\\	
	$\vec x\cdot\vec y$ & the dot product of vectors $\vec x$ and $\vec y$.\\	
	$\vec x^T$ & the transpose of vector $\vec x$.\\	
	$\overline{AB}$ & a segment between two points.\\	
	$2\nit, ~2\nit{+}1$ & the sets of even and odd numbers.\\	
	$f'_-(x), f'_+(x)$ & left and right derivatives of $f$.\\	
	$sign(x)$ & the sign function.
	\end{tabular}
	\section{Introduction}
	Convexity plays a central role in mathematical optimization. Under constraint qualification conditions \cite{wang2013constraint}, the Karush-Kuhn-Tucker (KKT) necessary optimality conditions become also sufficient for convex programs \cite{boyd2004convex}. In addition, convexity of the constraints is used to prove convergence (and rates of convergence) of specialized algorithms \cite{nesterov1994interior}. However, real-world problems often describe non-convex regions, and relaxing the convexity assumption while maintaining some optimality properties is highly desirable.
	
	One such property, called Kuhn-Tucker invexity, is the sufficiency of KKT conditions for global optimality:
	
	\begin{definition}\cite{martin1985essence}
		An optimization problem is said to be Kuhn-Tucker invex (KT-invex) if every KKT point is a global optimizer.
	\end{definition}
	
	Various notions of generalized convexity have been proposed in the literature. Early generalizations include pseudo- and quasi-convexity introduced by Mangasarian in \cite{mangasarian1965pseudo} where he also proves that problems with a pseudo-convex objective and quasi-convex constraints are KT-invex. Hanson \cite{hanson1981sufficiency} defined the concept of invex functions and gave a sufficient condition for KT-invexity, which was relaxed by Martin \cite{martin1985essence} in order to obtain a condition that is both necessary and sufficient. Later on, Craven \cite{craven1985invex} investigated the properties of invex functions.
	
	These ideas inspired more research on generalized convexity. K-invex \cite{craven1981invex}, preinvex \cite{ben1986invexity}, B-vex \cite{bector1991b}, V-invex \cite{jeyakumar1992generalised}, (p,r)-invex \cite{antczak2001p} and other types of functions and their roles in mathematical optimization.
	
%	Recently these ideas have been extended to nondifferentiable (), interval-valued problems (), 
	
	However, to the best of our knowledge, there are no computationally efficient procedures to check KT-invexity in practice even when restricted to two-dimensional spaces. To address this problem, we propose a new set of conditions expressed in terms of the behavior of the objective function on the boundary of the feasible set. We prove that these conditions are necessary and, for two-dimensional problems, sufficient for KT-invexity.
	
%	--------------------------
%	Invex functions were introduced by Hanson in \cite{Hanson1981545}. Hanson studied optimization problems where the objective function and all the constraints were given by differentiable functions satisfying
%	
%	\begin{equation}\label{invfunc}
%	f(x) - f(u) \geq \eta(x,u)^T\nabla f(u) ~\forall x,u \in C,\end{equation}
%	
%	where $C$ is a set ($C \in E^n$), the 'T' upper index denotes the transpose, the $\nabla$ symbol denotes the gradient and $\eta(x,u)$ is some vector-function. Hanson's paper showed that if such an $\eta(x,u)$ exists that it satisfies (\ref{invfunc}) for all the functions in a problem, then every Karush-Kuhn-Tucker (KKT) point in this problem is a global optimum.
%	
%	In \cite{craven1981duality}, Craven proved duality theorems for fractional programs with functions satisfying (\ref{invfunc}) and called such functions invex. Ben-Israel and Mond \cite{ben1986invexity} proved that a function is invex if and only if every stationary point is a global minimum. Mond and Hanson in \cite{mond1984duality} and Craven in \cite{craven1981invex} studied invexity of cones.
%	
%	Martin \cite{martin1985essence} proposed a relaxation of Hanson's conditions and showed that it is a necessary and sufficient condition for global optimality of all KKT points and introduced the notion of KT-invexity.
%	-------------------------
	The paper is organized as follows. In Section \ref{sec:KTinv} we introduce the notion of boundary-invexity and study its connection to the local optimality of KKT points. Here we also establish the connection between global optimality on the boundary and in the interior. Section \ref{section_crossprod} gives the definition of a two-dimensional cross product. In Section \ref{sec:bnd} we %look at the boundary of the feasible set, establish the connection between global optimality on the boundary and in the interior and 
	define a parametrization of the boundary curve. In Section \ref{sec:splitting} we study the behavior of concave functions on a line and present some results on boundary-optimality. Section \ref{sec:main} presents the main theorem establishing the sufficiency of boundary-invexity for two-dimensional problems. Finally, Section \ref{sec:app} investigates boundary-invexity of the Optimal Power Flow problem and Section \ref{sec:conclusion} concludes the paper.
	
	\section{Conditions for Kuhn-Tucker invexity}\label{sec:KTinv}
	
	Consider the optimization problem:
	
	\begin{align*}\label{nlpn}
	\text{max } & f(\vec x) \\
	\text{ s.t. } & g_{i}(\vec x) \leq 0 ~\forall i = 1..m \tag{NLP} \\
	%	& h^0_i(\vec x) = 0 ~\forall i = 1..k\\
	&\vec x \in \rit^{n},
	\end{align*}
	where all functions $f(\vec x)$, $g(\vec x)$ and $h(\vec x)$ are twice continuously differentiable and $f(\vec x)$ is concave. The results in this paper can be extended to problems with quasiconcave objective functions since only convexity of the superlevel sets of $f$ is used in the proofs.
	
	Let $F$ denote the feasible set of (\ref{nlpn}).
	
	\begin{definition}\cite{wright1999numerical}
		A solution $\vec x^*$ of problem (\ref{nlpn}) is said to satisfy Karush-Kuhn-Tucker (KKT) conditions if there exist constants $\mu_i ~(i = 1,...,m)$, called KKT multipliers, such that
		
		\begin{align}\label{kkt_conditions1}& \nabla f(\vec x^*) = \sum_{i=1}^{m} \mu_i\nabla g_i(\vec x^*), \\% + \sum_{j=1}^{k} \lambda_j\nabla h_j(\vec x^*), \\
		& g_i(\vec x^*) \leq 0, ~\forall i = 1,...,m, \label{kkt_conditions2}\\
		%		& h_j(\vec x^*) = 0, \text{ for } j = 1,...,k \label{kkt_conditions4}\\
		& \mu_i \geq 0,  ~\forall i = 1,...,m,\\
		& \mu_ig_i(\vec x) = 0,  ~\forall i = 1,...,m.\label{kkt_conditions3}
		\end{align}
	\end{definition}
	
	Points that satisfy KKT conditions are referred to as KKT points.
	\begin{definition}\cite{wright1999numerical}
		A point $\vec x^* \in \rit^n$ is a local maximizer for \eqref{nlpn} if $\vec x^* \in F$ and there is a neighborhood $N(\vec x^*)$ such that $f(\vec x) \leq f(\vec x^*)$ for $\vec x \in N(\vec x^*) \cap F$.
	\end{definition}

Let us emphasize that checking local optimality is NP-hard in general:
	
	\begin{theorem}\cite{pardalos1988checking}\label{th:nphard}
		The problem of checking local optimality for a feasible solution of (\ref{nlpn}) is NP-hard.
	\end{theorem}
	
In this work, we try to investigate necessary and sufficient conditions that allow us to circumvent the negative result presented in Theorem \ref{th:nphard} by identifying problems where KKT points are provably global optimizers.
	\subsection{Weak boundary-invexity}
	
%	Consider the problem (\ref{nlpn}) with $k=0$ (only inequality constraints). 
	For each non-convex constraint $g_i(x)\le 0$ % let $E_i \subset \{1..m\}\setminus {i}$ denote the set of indices corresponding to all possible subsets of constraints excluding $i$ and 
	define the problem:
	
%	For each non-convex constraint $g^0_i$ define a problem:
	
%	\begin{align*}\label{nlpin}
%	& min ~f^0(\vec x) \tag{NLPi} \\
%	& s.t. ~g^0_i(\vec x) \geq 0
%	\end{align*} 

	\begin{align*}\label{nlpin}
	& \min ~f(\vec x) \tag{NLP$_{i}$} \\
	& s.t. ~g_i(\vec x) = 0.
	\end{align*} 
	
	%The number of such problems per one concave constraint is thus $\sum\limits_{k=1}^p C_p^k$, where $p=\min\{|I^i|,n-1\}$, $|I^i|$ denotes the cardinality of the set $I^i$ and $C_p^k$ is the binomial coefficient. Although this number is exponential in $p$, in many practical applications each constraint shares the same variables with few other constraints and thus $p$ is a small number.
	
%	Let $A_i(\vec x)$ denote the set of constraints in (\ref{nlpin}) that are active at $\vec x$.
	
	\begin{definition} (Weak boundary-invexity)
		Problem (\ref{nlpn}) is weakly boundary-invex if (\ref{nlpin}) is unbounded or at least one of the following holds for its global minimum $\vec x^*$:

	\begin{enumerate}
		\item $\vec x^*$ is infeasible for (\ref{nlpn}),
		\item $\vec x^*$ is not a strict minimizer,
		\item the KKT multiplier for $\vec x^*$ in (\ref{nlpin}) is non-negative,
		\item there exist constraints $g_j(\vec x) \le 0,~j\ne i$ in (\ref{nlpn}) that are active at $\vec x^*$.
	\end{enumerate}		
		%all KKT points of (\ref{nlpin}) with non-positive KKT multipliers are infeasible or represent local maxima with respect to (\ref{nlpn}). 
		%that are either saddle points or strict local minima are not in $F$ or are local maxima of (\ref{nlpn}).%all constraints $g_i \leq 0$, $i \in A(\vec x^*) ~|~ \mu_i > 0$ are redundant for (\ref{nlpn}) in some neighborhood of $\vec x^*$ excluding $\vec x^*$ itself.
	\end{definition}

	%These points always lie on the boundary of the feasible set of (\ref{nlpin}) since a concave function does not have any local unconstrained minima or saddle points.

%	Only those solutions of (\ref{nlpin}) where all constraints are active ($E_i(\vec x) = E_i$) can belong to $F$, because otherwise $g_j(\vec x^*) > 0$ for some $j$ and thus $\vec x^*$ is not a feasible solution of (\ref{nlpn}). Then finding all solutions of (\ref{nlpin}) that can potentially violate boundary-invexity is equivalent to considering the following inequality-constrained problem:
%	
%	\begin{align*}
%	& min ~f(\vec x) \\
%	& s.t. ~g_j(\vec x) = 0, ~\forall j \in E_i
%	\end{align*} 
%	and finding all solutions with non-negative KKT multipliers.
	
		(\ref{nlpin}) is still a non-convex problem, and finding its global optimum can be NP-hard in general. However, in some special cases (\ref{nlpin}) can be more tractable than \eqref{nlpn} since we are restricting the feasible region to one of its boundaries.
%	For example, if functions $g_j$ are concave, then (\ref{nlpin}) becomes a problem of minimizing a concave function over a convex set.

	For instance, when both $f(\vec x)$ and $g_i(\vec x)$ are quadratic functions we can apply an extension of the S-lemma:

	\begin{theorem}\cite{Xia2016}
		Let $f(\vec x) = \vec x^TA\vec x + a^T\cdot\vec x + c$ and $g(\vec x) = \vec x^TB\vec x + b^T\cdot\vec x + d$ be two quadratic functions having symmetric matrices $A$ and $B$. If $g(\vec x)$ takes both positive and negative values and $B\neq 0$, then the following two statements are equivalent:

		\begin{enumerate}
			\item $(\forall \vec x \in \rit^n) ~g(\vec x) = 0 \implies f(\vec x) \geq 0$,
			\item There exists a $\mu \in \rit$ such that $f(\vec x) + \mu g(\vec x) \geq 0, ~\forall \vec x \in \rit^n$.
		\end{enumerate}
	\end{theorem}	
	Using this theorem and based on the approach described in \cite{Xia2016}, (\ref{nlpin}) can be reformulated as a Semidefinite Program and thus solved efficiently.
	
	\subsection{Necessary condition for KT-invexity}
	
	\begin{theorem} (Necessary condition)
		If (\ref{nlpn}) is KT-invex, then it is weakly boundary-invex.
	\end{theorem}
	\begin{proof}
		We will proceed by contradiction, assume that (\ref{nlpn}) is KT-invex but not weakly boundary-invex. Thus, there exists a point $\vec x^* \in F$ which is a global minimizer and therefore a KKT point of (\ref{nlpin}):
		
		\begin{align*}& \nabla f(\vec x^*) = -\lambda_i\nabla g_i(\vec x^*), \\
		& g_i(\vec x^*) = 0,\\
		& \lambda_i < 0.
		\end{align*}
		
		Let $\mu_i = -\lambda_i$. Since $g_i$ is the only active constraint at $\vec x^*$, we can set $\mu_j = 0, ~j=1,...,i-1,i+1,...,m$ and obtain the following system:	
		\begin{align*}& \nabla f(\vec x^*) = \sum_{j=1}^m \mu_j\nabla g_j(\vec x^*), \\
		& g_j(\vec x^*) = 0, ~\forall j=1,...,m,\\
		& \mu_j \geq 0, ~\forall j=1,...,m,
		\end{align*}
		implying that $\vec x^*$ is a KKT point of (\ref{nlpn}).		
%		Since $\vec x^*$ is the global minimum for (\ref{nlpin}), we have 
%		$$f(\vec x) \geq f(\vec x^*) ~\forall \vec x \text{ s.t. } g_i(\vec x) = 0.$$ 
		Since no other constraints are active at $\vec x^*$, there exists a point $\hat{\vec x}$ in the neighborhood of $\vec x^*$, such that 
		$$g_i(\hat{\vec x}) = 0 \text{ and } \hat{\vec x} \in F$$
		Since $\vec x^*$ is a strict global minimizer in \eqref{nlpin}, we have that $f(\vec x^*) < f(\hat{\vec x})$ which contradicts with (\ref{nlpn}) being KT-invex.

%Thus, there exists a point $\vec x^* \in F$ which is not a local maximizer in \eqref{nlpn} and is a KKT point of (\ref{nlpin}):% $E_i \subset I^i$
%		
%%		\begin{align*}& \nabla f(\vec x^*) = \sum_{j \in E_i} \mu_j\nabla g_j(\vec x^*), \\
%%		& g_j(\vec x^*) \geq 0, ~\forall j \in E_i,\\
%%		& \mu_j \geq 0, ~\forall j \in E_i, \\
%%		& \mu_jg_j(\vec x) = 0, ~\forall j \in E_i.
%%		\end{align*}
%%		
%%		This can be rewritten as:
%		
%		\begin{align*}& \nabla f(\vec x^*) = \sum_{j\in E_i(\vec x^*)} \mu_j\nabla g_j(\vec x^*), \\
%		& g_j(\vec x^*) = 0, ~\forall j \in E_i(\vec x^*),\\
%		& \mu_j \geq 0, ~\forall j \in E_i(\vec x^*).
%		\end{align*}
%		
%		Since $\vec x^* \in F$ and for all $i \notin E_i(\vec x^*)$ we can set $\mu_i = 0$, the following holds:
%		
%		\begin{align*}& \nabla f(\vec x^*) = \sum_{j=1}^{m} \mu_j\nabla g_j(\vec x^*), \\
%		& g_j(\vec x^*) \leq 0, ~\forall j=1,...,m,\\
%		& \mu_j \geq 0, ~\forall j=1,...,m, \\
%		& \mu_jg_j(\vec x) = 0, ~\forall j=1,...,m.
%		\end{align*}
%		
%		Thus $\vec x^*$ is a KKT point of (\ref{nlpn}) that is not a local maximizer, which implies that (\ref{nlpn}) is not KT-invex, contradiction.
	\end{proof}
	\qed
	
%	\section{Boundary of the feasible set}\label{sec:bnd}
	\subsection{Connection between boundary and interior optimality}
%Assume that all $g_i(x_1,x_2)$ in \ref{nlp} are either convex or concave. 
%Consider a non-empty set $F_n \subset \rit^n$ and a continuous concave function $f: \rit^n \rightarrow \rit$.
%Let $F$ be the feasible region for \eqref{nlp} and $\partial F$ be the boundary of $F$. We will assume that $F$ is a bounded set.

	\begin{definition}\cite{simmons1963introduction}
		A connected set is a set which cannot be represented as the union of two disjoint non-empty closed sets.
	\end{definition}
	
	\begin{lemma}\label{lemma_bnd}
		Given a local maximizer $\vec x^* \in \rit^n$ for \eqref{nlpn}, if $F$ is connected then the following statement is true:\\
		If $\vec x^*$ is a global maximizer on $\partial F$ then it is also a global maximizer for \eqref{nlpn}. 
	\end{lemma}
	\begin{proof}
		$\vec x^*$ is a local maximizer, so there is a neighborhood $N(\vec x^*)$ such that if $f(\vec x) > f(\vec x^*)$ and $\vec x \in N(\vec x^*)$, then $\vec x \notin F$.
		
		Let us prove the lemma by contradiction. Consider an arbitrary point $\hat{\vec x} \in F$ such that $f(\hat{\vec x}) > f(\vec x^*)$. Since $f$ is concave, there exists a convex set $L_c(f) = \{\vec x ~|~ f(\vec x) \geq c\}$, where $c$ satisfies $f(\vec x^*) < c < f(\hat{\vec x})$. Since $f$ is continuous, $c$ can be chosen so that $\partial L_c(f) \cap N(\vec x^*)$ is non-empty.
		Note that $\hat{\vec x} \in L_c(f)$.
		
		Since $f(\vec x) \leq f(\vec x^*) ~\forall \vec x \in \partial F$ and $f(\vec x) > f(\vec x^*) ~\forall \vec x \in \partial L_c(f)$, the two boundaries cannot have common points: $\partial L_c(f) \cap \partial F = \emptyset$. Given that $F$ is connected, there are three possibilities: \\
		
		1) If $F \cap L_c(f)  = \emptyset$. Contradiction, since $\hat{\vec x} \in L_c(f)$ would imply that $\hat{\vec x} \notin F$.
		
		2) If $F \subset L_c(f)$. Contradiction, since $\vec x^* \in F$ and $\vec x^* \notin L_c(f)$ given that $f(\vec x^*) < c$.
		
		3) If $L_c(f) \subset F$. Given that $\partial L_c(f) \cap N(\vec x^*)$ is non-empty, points in this intersection have a higher objective function value with respect to $\vec x^*$ and belong to its neighborhood are feasible. This contradicts with $\vec x^*$ being a local maximizer.
		%In this case, since $f(\vec x) > f(\vec x^*) \Rightarrow \vec x \notin F$ for all $\vec x \in N(\vec x^*)$ and $f(\vec x) > f(\vec x^*) ~\forall \vec x \in \partial L_c(f)$, there are infeasible points in $\partial L_c(f)$. 
		%Contradiction.
		
		We have proven that $\hat{\vec x} \notin F$ for any $\hat{\vec x}$ such that $f(\hat{\vec x}) > f(\vec x^*)$. Thus $\vec x^*$ is a global maximizer in $F$.
	\end{proof}
	\qed

	%Generally, checking the boundary-invexity conditions is a hard problem. However, if the objective and all non-convex constraints are given by quadratic functions, then %then we can use the result from \cite{more1993generalizations}. It deals with problems of the form:
	
%	\begin{align*}\label{quad_prob}
%	& min ~\vec x^T\vec A\vec x + 2\vec a^T\vec x + \vec c \tag{QNLPi} \\
%	& s.t. ~\vec x^T\vec B\vec x + 2\vec b^T\vec x + \vec d = 0,
%	\end{align*}
%	
%	where $\vec A$ and $\vec B$ are symmetric matrices.
	
%	\begin{theorem}\cite{more1993generalizations}
%		If $\vec B \neq 0$ and $\vec x^T\vec B\vec x + 2\vec b^T\vec x + \vec d$ takes both positive and negative values, a vector $\vec x^*$ is a global minimizer 
%	\end{theorem}
	
%	Then problem (\ref{nlpin}) is KT-invex and for checking the boundary-invexity condition it is sufficient to find and check a KKT point of this problem.
	
%	\cite{more1993generalizations}
	
	\subsection{Problems with two degrees of freedom}
	To the best of our knowledge, there are no polynomial-time verifiable necessary and sufficient conditions for checking KT-invexity even in two dimensions.
	In this work, we try to take a first step in this direction, showing that boundary-invexity is both necessary and sufficient while being efficiently verifiable.
	Even after restricting the problem to two degrees of freedom, the proof of sufficiency is not straightforward and requires an elaborate geometric reasoning. 
	In the following sections, we try to brake up our approach into various pieces, in the hope of making it easier for the reader.
	
	We consider the following optimization problem:
	\begin{align*}\label{nlp0}
	\text{max } & f^0(\vec x) \\
	\text{ s.t. } & g^0_{i}(\vec x) \leq 0 ~\forall i = 1..m \tag{NLP$_0$} \\
	& h^0_i(\vec x) = 0 ~\forall i = 1..n-2\\
	&\vec x \in \rit^{n}.
	\end{align*}
	
	and assume that $n-2$ variables can be projected out given the system of non-redundant $n-2$ linear equations $h^0_i(\vec x) = 0$. After projecting these variables out, (\ref{nlp0}) can be expressed as a two-dimensional problem:
	
	\begin{align*}\label{nlp}
	\text{max } &f(x_1,x_2) \\
	\text{ s.t. } &g_{i}(x_1,x_2) \leq 0 \text{ } \forall i = 1..m \tag{NLP$_2$} \\
	&(x_1,x_2) \in \rit^{2}.
	\end{align*}

	\begin{definition}\cite{krantz2002primer}
		A real function $f$ is said to be real analytic at $\vec x^0$ if it may be represented by a convergent power series on some interval of positive radius centered at $\vec x^0$:
		
		$$f(\vec x) = \sum\limits_{j=0}^{\infty}a_j(\vec x - \vec x^0)^j$$
		
		The function is said to be real analytic on a set $S \subset \rit^n$ if it is real analytic at each $\vec x^0 \in S$.
	\end{definition}
	
	We will assume that $f$ is a concave real analytic function, $g_i$ are twice continuously differentiable, $F$ is connected and bounded and LICQ holds for all points $\vec x \in \partial F$.

	Given these assumptions, the corresponding boundary-invexity models (\ref{nlpin}) become:
	
	\begin{align*}\label{nlpi}
	& \min ~f(x_1,x_2) \tag{NLP$_{2i}$} \\
	& s.t. ~g_i(x_1,x_2) \geq 0.
	\end{align*}
	
	We will define a stronger version of the boundary-invexity property, which is both necessary and sufficient for KT-invexity of (\ref{nlp}):

	\begin{definition} (Boundary-invexity)\label{def:bi}
		Problem (\ref{nlp}) is boundary-invex if at least one of the following holds for all KKT points $\vec x^*$ of (\ref{nlpi}):
		
		\begin{enumerate}
			\item $\vec x^*$ is infeasible for (\ref{nlp}),
			\item $\vec x^*$ has non-negative KKT multipliers in (\ref{nlpin}),
			\item $\vec x^*$ is a local maximum with respect to (\ref{nlp}). 
		\end{enumerate}
	\end{definition}

	%	\begin{definition}\cite{wright1999numerical}
	%		Given the point $\vec x$ and a set of active constraint indices $A(\vec x)$, we say that the linear independence constraint qualification (LICQ) holds at $\vec x$ if the set of active constraint gradients $\{\nabla g_i(\vec x)$, $i \in A(\vec x)\}$ is linearly independent.
	%	\end{definition}

%	Those solutions $\vec x^*$ in whose neighborhood $g_i$ is convex do not violate boundary-invexity since they are local maxima of (\ref{nlp}). Because of this, we need to consider only such solutions where $g_i$ is either locally concave or has an inflection point. Both are necessary for ensuring that KKT points of (\ref{nlp}) are locally optimal, and the former are of a special interest to us because, as we will show in the sufficiency proof, multiple local maxima can occur only if such points exist in (\ref{nlp}).
	
	\subsection{Local optimality of KKT points}
	
	We first recall a result from \cite{wright1999numerical}. Let $A(\vec x)$ be the set of all active constraints at point $\vec x$.
	
	\begin{definition}\label{critical_cone}
		Given a KKT point $\vec x^*$ of problem (\ref{nlp}) and corresponding Lagrange multiplier vector $\vec{\mu}$, a critical cone $C(\vec x^*,\vec \mu)$ is defined as a set of vectors $\vec w$ such that:
		
		$$\begin{cases}
		(\nabla g_i(\vec x^*))^T\cdot\vec w=0 ~\forall i ~|~ g_i(\vec x^*)=0, ~\forall i \in A(\vec x^*) \text{ with } \vec \mu_i>0,\\
		(\nabla g_i(\vec x^*))^T\cdot\vec w\leq 0 ~\forall i ~|~ g_i(\vec x^*)=0, ~\forall i \in A(\vec x^*) \text{ with } \vec \mu_i=0.
		\end{cases}$$
	\end{definition}
	%mu = 0 corresponds to the constraint which could be removed and the point would still be KKT. w needs to be just feasible with respect to these.
	%mu > 0 corresponds to the important constraints. Need to stay on these
	
	The directions contained in the critical cone are important for distinguishing between a local maximum and other types of stationary points.
	
	%	The critical cone contains the directions $w$ for which it is impossible to tell from first order derivative information only whether the objective is decreasing or increasing. Indeed, from Definition \ref{critical_cone} and the fact that $\mu_i = 0 ~\forall i \in I \setminus A(\vec x^*)$ we have:
	
	%	$$w \in C(\vec x^*,\vec \mu) ~\Rightarrow~ \mu_i\nabla g_i^T(\vec x^*)\vec w = 0 ~\forall i \in I$$
	
	%	From the KKT condition, it follows that
	%	
	%	$$w \in C(\vec x^*,\vec \mu) ~\Rightarrow~ \vec w^T\nabla f(\vec x^*) = \sum_{i \in I}\mu_i\nabla g_i^T(\vec x^*)\vec w = 0.$$
	
	\begin{theorem}\label{theorem_sosc}\cite{wright1999numerical}(Second-order sufficient conditions)
		Let $\vec x^*$ be a KKT point for problem (\ref{nlp}) with a Lagrange multiplier vector $\vec \mu$. Suppose that
		$$\vec w^T\nabla^2_{\vec x}L(\vec x^*,\vec \mu)\vec w > 0 ~\forall \vec w \in C(\vec x^*,\vec \mu), ~\vec w \neq 0,$$
		where $L(\vec x, \vec \mu) = f(\vec x) - \sum\limits_{i=1}^m\mu_ig_i(\vec x)$ is the Lagrangian function.
		
		Then $\vec x^*$ is a strict local maximum in (\ref{nlp}).
	\end{theorem}
	
	\begin{lemma}\label{lemma_kkt_locmax} Suppose that (\ref{nlp}) is boundary-invex. Then every KKT point is a local maximum.
	\end{lemma}	
	\begin{proof}
		Consider a KKT point $\vec x^*$.		
		Since (\ref{nlp}) is two-dimensional, at most two constraints can be active and non-redundant at $\vec x^*$. Let these constraints be denoted as $g_1$ and $g_2$ and let the corresponding KKT multipliers be $\mu_1$, $\mu_2$.
		
		\begin{enumerate}
			\item If both $\mu_i > 0$, then the critical cone can be written as:
			
			$$\vec w \in C(\vec x^*,\vec \mu) ~\Leftrightarrow~
			\begin{cases}
			(\nabla g_1(\vec x^*))^T\cdot\vec w = 0\\
			(\nabla g_2(\vec x^*))^T\cdot\vec w = 0
			\end{cases}
			$$
			
			$$~\Rightarrow~ 
			\begin{sqcases}
			\vec w = 0\\
			(\nabla g_1(\vec x^*))^T\cdot\vec w = (\nabla g_2(\vec x^*))^T\cdot\vec w 
			\end{sqcases}
			~\Rightarrow~
			\begin{sqcases}
			\vec w = 0\\
			\nabla g_1(\vec x^*) = \nabla g_2(\vec x^*)
			\end{sqcases}$$
			
			In the first case, $\vec w \in C(\vec x^*,\vec \mu) ~\Leftrightarrow~ \vec w = 0$. The conditions of Theorem \ref{theorem_sosc} are satisfied and $\vec x^*$ is a local maximum. Otherwise LICQ is violated.
			
			\item Suppose that $\mu_2 = 0$ and $\mu_1 > 0$. Then, by (\ref{kkt_conditions1}), $\nabla f(\vec x^*) = \mu_1\nabla g_1(\vec x^*)$.  Then the following cases are possible:
			
			\begin{enumerate}
				
				\item $g_1$ is convex. Since (\ref{kkt_conditions2}) and (\ref{kkt_conditions3}) are satisfied, $\vec x^*$ is a KKT point for a problem of maximizing $f$ on $g_1(\vec x) = 0$. Then it is a local maximum for this problem and, since it is a relaxation of (\ref{nlp}), a local maximum for (\ref{nlp}).
				
				\item $g_1$ is non-convex. Setting $\lambda_1 = -\mu_1$, we get $\nabla f(\vec x^*) = -\lambda_1\nabla g_1(\vec x^*)$, $\lambda_1 < 0$. (\ref{kkt_conditions3}) implies that $g_1(\vec x^*) = 0$. Then $\vec x^*$ is a KKT point for (\ref{nlpi}) with a negative KKT multiplier which is feasible for (\ref{nlp}). Since (\ref{nlp}) is boundary-invex, $\vec x^*$ is a local maximum.
				
%				\item There exists a neighborhood of $\vec x^*$ where $g_1$ is convex. Then $\vec x^*$ is a local maximum in (\ref{nlp}).
				
%				\item There exists a neighborhood of $\vec x^*$ where $g_1$ is strictly concave. Then $\vec x^*$ is a local minimum of $f$ on $g_1(\vec x) = 0$. This contradicts with boundary-invexity of (\ref{nlp}).
%				
%				\item For any neighborhood of $\vec x^*$ the function $g_1$ is neither convex nor concave. Then $\vec x^*$ is an inflection point of $g_1$. This contradicts with boundary-invexity of (\ref{nlp}).
			\end{enumerate}
			
			\item $\mu_1 = \mu_2 = 0$. Then $\vec x^*$ is the unconstrained global maximum of $f$ and thus a maximum for (\ref{nlp}).
			
		\end{enumerate}
	\end{proof}
	\qed
	
	\section{Two-dimensional cross product}\label{section_crossprod}
	
	\begin{definition} Given two vectors $\vec x, \vec y \in \rit^2$ define their cross product to be 
	$$\vec x \times \vec y = x_1y_2 - x_2y_1.$$
	\end{definition}
	
	The sign of $\vec x \times \vec y$ has a geometric interpretation. If $\vec x \times \vec y > 0$, then the shortest angle at which $\vec x$ has to be rotated for it to become co-directional with $\vec y$ corresponds to a counter-clockwise rotation. If $\vec x \times \vec y < 0$, then such an angle corresponds to a clockwise rotation. If $\vec x \times \vec y = 0$, the vectors are parallel.

	\begin{definition}[Tangent vector]\cite{abbena2006modern}
		Given a parametrization $(x_1(t),x_2(t))$ of a curve $g(x_1,x_2) = 0$, the vector $(x_1'(t),x_2'(t))^T$ is said to be its tangent vector.
	\end{definition}
	
	Tangent vectors are orthogonal to gradient vectors. This can be proven using the chain differentiation rule:
	
	$$g(x_1(t),x_2(t)) = 0 ~\Rightarrow ~\frac{\partial g}{\partial x_1}\frac{\partial x_1}{\partial t} + \frac{\partial g}{\partial x_2}\frac{\partial x_2}{\partial t} = (\nabla g(x_1,x_2))^T\cdot(x_1'(t),x_2'(t))^T = 0.$$
	
	\begin{lemma}
		Given a differentiable function $g: \rit^2 \rightarrow \rit$, a point $\vec y = (y_1,y_2)$ such that $g(\vec y) = 0$, the vector $(-g'_{x_2}(\vec y),g'_{x_1}(\vec y))$ is the tangent vector to the curve $g(\vec x) = 0$ at point $\vec y$.
	\end{lemma}
	\begin{proof}
	Considering the dot product,
	$$(-g'_{x_2}(\vec y),g'_{x_1}(\vec y))\cdot\nabla g(\vec y) = (-g'_{x_2}(\vec y),g'_{x_1}(\vec y))\cdot(g'_{x_1}(\vec y),g'_{x_2}(\vec y))^T = $$
	$$= -g'_{x_2}(\vec y)g'_{x_1}(\vec y) + g'_{x_1}(\vec y)g'_{x_2}(\vec y) = 0$$
	the vector $(-g'_{x_2}(\vec y),g'_{x_1}(\vec y))$ is orthogonal to the gradient and thus a tangent to the curve $g(\vec x)=0$ at the point $\vec y$. %This vector shows the positive direction of moving along this curve (keeping the set $g(\vec x) \leq 0$ on the left), and the opposite vector shows the negative direction of moving along the curve (keeping the set $g(\vec x) \leq 0$ on the right).
	\end{proof}
	\qed
	
	\begin{figure}[h]
		\includegraphics[scale=0.4]{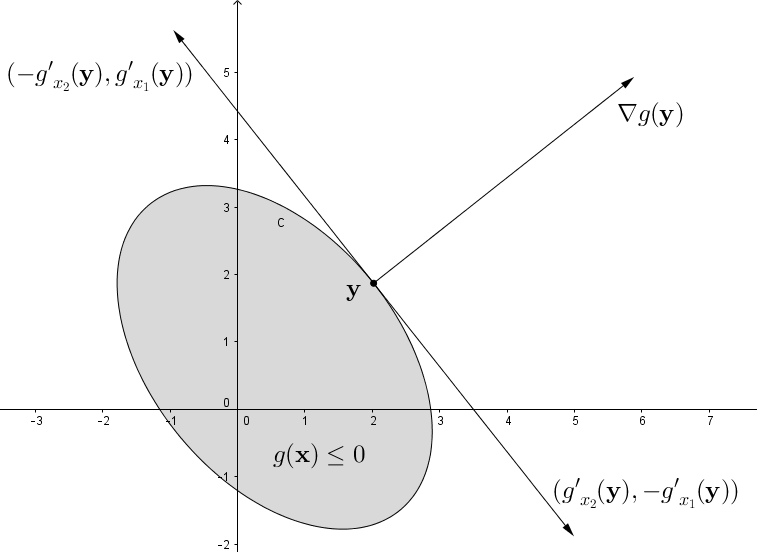}
		\caption{Tangent vectors}
	\end{figure}
	
	\begin{definition}
		The positive (resp. negative) direction of moving along the curve $g(\vec x) = 0$ is the direction corresponding to the vector $(-g_{x_2}(\vec x),g_{x_1}(\vec x))$ (resp. $(g_{x_2}(\vec x),-g_{x_1}(\vec x))$).
	\end{definition}
	
%	Moving along $g(\vec x) = 0$ in the positive direction keeps the set $g(\vec x) \leq 0$ on the left. Moving along this curve in the negative direction keeps  the set $g(\vec x) \leq 0$ on the right.
	
	\begin{definition}
		\cite{wrede2010advanced}		
	Given a differentiable function $f$, the directional derivative of $f$ along vector $\vec u$ is defined as:	
	$$\frac{\partial f(\vec x)}{\partial \vec u} = (\nabla f(\vec x))^T\cdot\vec u$$
	\end{definition}	
	\begin{lemma}\label{lemma_cross_prod}
		Consider differentiable functions $f: \rit^2 \rightarrow \rit$ and $g: \rit^2 \rightarrow \rit$. We have $\nabla f(\vec y) \times \nabla g(\vec y) \geq 0$ (resp. $\nabla f(\vec y) \times \nabla g(\vec y) \leq 0$) if and only if $f(\vec x)$ is non-increasing (resp. non-decreasing) when moving along the curve $g(\vec x) = 0$ in the positive direction.
	\end{lemma}
	\begin{proof}
		We will prove the case where $f$ is non-increasing.
		
		Consider the directional derivative of $f$ with respect to the tangent vector at point $\vec y$:
		
		$$\frac{\partial f}{\partial (-g'_{x_2}(\vec y),g'_{x_1}(\vec y))}(\vec y) = (\nabla f(\vec y))\cdot(-g'_{x_2}(\vec y),g'_{x_1}(\vec y))^T =  $$
		$$-f'_{x_1}(\vec y)g'_{x_2}(\vec y) + f'_{x_2}(\vec y)g'_{x_1}(\vec y) = -\nabla f(\vec y) \times \nabla g(\vec y) \leq 0,$$
		and this implies that the cross product being non-negative at $\vec y$ is equivalent to $f$ being non-increasing on $g(\vec x) = 0$ at $\vec y$.
	\end{proof}
	\qed
	
	\subsection{Reformulation of the KKT conditions}
	
	Now we shall establish a connection between the KKT conditions and the sign of the cross products corresponding to the gradient vectors.
	
	\begin{lemma}\label{lemma:kkt_prod_signs}
		Consider a point $\vec x^* \in F$ with two active non-redundant constraints $g_1(\vec x) \leq 0$ and $g_2(\vec x) \leq 0$ such that $\nabla g_1(\vec x^*) \times \nabla g_2(\vec x^*) > 0$. $\vec x^*$ is a KKT point if and only if 
		
		\begin{align*} & \nabla f(\vec x^*) \times \nabla g_1(\vec x^*) \geq 0, \\
		& \nabla f(\vec x^*) \times \nabla g_2(\vec x^*) \leq 0. \end{align*}
	\end{lemma}
	\begin{proof}
		By KKT conditions (\ref{kkt_conditions1})-(\ref{kkt_conditions3}), there exist $\mu_1, \mu_2$ such that the following holds:
		
		\begin{equation*}
		\begin{cases} \label{kkt}
		\mu_{1} \frac{\partial g_1}{\partial x_1}(\vec x^*) + \mu_{2} \frac{\partial g_2}{\partial x_1}(\vec x^*) = \frac{\partial f}{\partial x_1}(\vec x^*) 
		\\[2mm]
		\mu_{1} \frac{\partial g_1}{\partial x_2}(\vec x^*) + \mu_{2} \frac{\partial g_2}{\partial x_2}(\vec x^*) = \frac{\partial f}{\partial x_2}(\vec x^*) 
		\\[2mm]
		\mu_1 , \mu_2 \geq 0
		\end{cases}
		\end{equation*}
		
		From this system we can find $\mu_1 , \mu_2$:
		
		\begin{gather*}
		\mu_1 = \frac{\frac{\partial f}{\partial x_1}(\vec x^*) \frac{\partial g_2}{\partial x_2}(\vec x^*) - \frac{\partial g_2}{\partial x_1}(\vec x^*) \frac{\partial f}{\partial x_2}(\vec x^*)}{\frac{\partial g_1}{\partial x_1}(\vec x^*) \frac{\partial g_2}{\partial x_2}(\vec x^*) - \frac{\partial g_2}{\partial x_1}(\vec x^*) \frac{\partial g_1}{\partial x_2}(\vec x^*)} = %\nonumber \\
		\frac{\nabla f(\vec x^*) \times \nabla g_2(\vec x^*)}{\nabla g_1(\vec x^*) \times \nabla g_2(\vec x^*)} \label{mu1_2} \\[3mm]
		\mu_2 = \frac{\frac{\partial g_1}{\partial x_1}(\vec x^*) \frac{\partial f}{\partial x_2}(\vec x^*) - \frac{\partial f}{\partial x_1}(\vec x^*) \frac{\partial g_1}{\partial x_2}(\vec x^*)}{\frac{\partial g_1}{\partial x_1}(\vec x^*) \frac{\partial g_2}{\partial x_2}(\vec x^*) - \frac{\partial g_2}{\partial x_1}(\vec x^*) \frac{\partial g_1}{\partial x_2}(\vec x^*)} = %\nonumber \\
		\frac{\nabla g_1(\vec x^*) \times \nabla f(\vec x^*)}{\nabla g_1(\vec x^*) \times \nabla g_2(\vec x^*)} \label{mu2_2}
		\end{gather*}
		
		%together with equations (\ref{kkt}) this implies that $\nabla g_1(\vec x^*) = c_1\nabla f(\vec x^*)$ and $\nabla g_2(\vec x^*) = c_2\nabla f(\vec x^*)$. If both constraints are concave, then either $\vec x^*$ is not a KKT point or $\vec x^*$ is a solution of (\ref{nlpi}) and contradicts with the condition that all such points should be infeasible. So at least one of the constraints has to be convex, and $\vec x^*$ should be it's local maximizer. But then $\vec x^*$ is the global maximizer for a relaxation of (\ref{nlp}) which is obtained by removing all constraints except for $g_1$ and $g_2$, and, since this point feasible in (\ref{nlp}), it is the global maximizer in (\ref{nlp}). \\
		
		$\mu_1 , \mu_2 \geq 0$ is equivalent to
		\begin{gather*}
		\nabla f(\vec x^*) \times \nabla g_1(\vec x^*) \geq 0 \\
		\nabla f(\vec x^*) \times \nabla g_2(\vec x^*) \leq 0
		\end{gather*}
	\end{proof}
	\qed

	\section{Parametrization of the boundary of $F$}\label{sec:bnd}
	
	%Consider a point $\vec x$ such that $g_i(\vec x)=0$ and the tangent vector at $\vec x$: $(-g'_{ix_2}(\vec x),g'_{ix_1}(\vec x))$. Recall from Section \ref{section_crossprod} that this vector shows the positive direction of moving along the boundary curve (keeping the set $g_i(\vec x) \leq 0$ on the left).
	
	Given a real variable $t \in [0 , T]$, where $T \in \rit,~ T > 0$, define a parametrization $\gamma: \rit \rightarrow \rit^2$ of $\partial F$ such that $\gamma(0)=\gamma(T)$ and the direction of increase of $t$ corresponds to the positive direction of moving along the boundary. Then
	\begin{gather*}
	\gamma'_-(t) = \left(-\frac{\partial g_{i^-(t)}}{\partial x_2}(\gamma(t)), \frac{\partial g_{i^-(t)}}{\partial x_1}(\gamma(t))\right)^T\\
	\gamma'_+(t) = \left(-\frac{\partial g_{i^+(t)}}{\partial x_2}(\gamma(t)), \frac{\partial g_{i^+(t)}}{\partial x_1}(\gamma(t))\right)^T
	\end{gather*}
	where $i^-(t)$ and $i^+(t)$ are indices of constraints that are active at $\gamma(t)$ and non-redundant in some neighborhood of this point. If there is only one active non-redundant constraint at $\gamma(t)$, then $i^-(t)=i^+(t)=i(t)$ and $\gamma'_-(t)=\gamma'_+(t)=\gamma'(t)$. Otherwise we will require that there exists an $\epsilon_0>0$ such that $i^-(t) = i(t-\epsilon)$ and $i^+(t) = i(t+\epsilon) ~\forall \epsilon \in (0,\epsilon_0)$.
	
	Let $\gamma^r(t)$ be the reversed direction parametrization of $\partial F$:
	
	\begin{gather*}
	\gamma^{r\prime}_-(t) = \left(\frac{\partial g_{i^{r-}(t)}}{\partial x_2}(\gamma(t)), -\frac{\partial g_{i^{r-}(t)}}{\partial x_1}(\gamma(t))\right)^T,\\
	\gamma^{r\prime}_+(t) = \left(\frac{\partial g_{i^{r+}(t)}}{\partial x_2}(\gamma(t)), -\frac{\partial g_{i^{r+}(t)}}{\partial x_1}(\gamma(t))\right)^T,\\
	\end{gather*}
	where $i^{r-}(t)$, $i^{r+}(t)$ are defined in a similar way to the indices in the direct parametrization.
	
	\begin{figure}[h]
		\includegraphics[scale=0.4]{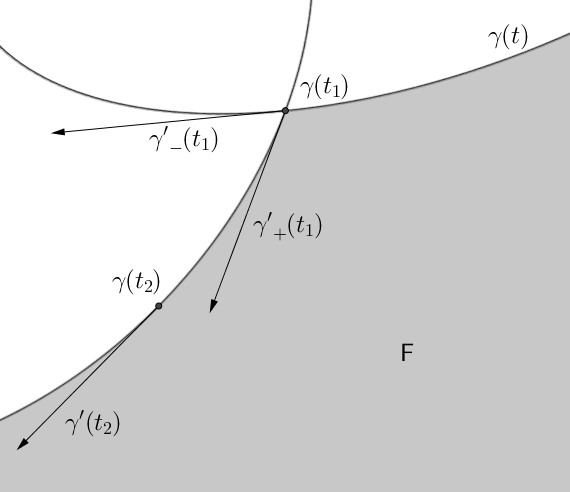}
		\caption{Parametrisation of the boundary of the feasible region}
	\end{figure}
	In the following Lemma, we show that $\gamma$ does not intersect itself.
	\begin{lemma}\label{lemma_no_self_intersect}
	Consider two distinct values $t_1$ and $t_2$ of parameter $t$, such that $0 < t_1 < t_2 < T$, then $\gamma(t_1) \neq \gamma(t_2)$.
%	, where $\gamma$ is the parametrization of $\partial F$.
	\end{lemma}
	\begin{proof}
	We will proceed by contradiction, suppose that there exist numbers $t_1$, $t_2$ such that $\gamma(t_1) = \gamma(t_2) = \vec y$ and $0 < t_1 < t_2 < T$. Let $j = i(t_1)$ and $k = i(t_2)$. Consider the product $(\nabla g_j(\vec y))^T\cdot \gamma'(t_2)$.
	\begin{enumerate}
	\item $(\nabla g_j(\vec y))^T\cdot\gamma'(t_2) = 0$. Then
	
	$$(g_j)'_{x_1}(\vec y)(g_k)'_{x_2}(\vec y) - (g_j)'_{x_2}(\vec y)(g_k)'_{x_1}(\vec y) $$
	$$= (\nabla g_j(\vec y))^T\cdot((g_k)'_{x_2}(\vec y),-(g_k)'_{x_1}(\vec y))^T= (\nabla g_j(\vec y))^T\cdot\gamma'(t_2) = 0$$
	and thus
	$$(g_j)'_{x_2}(\vec y) = \frac{(g_j)'_{x_1}(\vec y)(g_k)'_{x_2}(\vec y)}{(g_k)'_{x_1}(\vec y)}.$$
	
	If $c = -\frac{(g_j)'_{x_1}(\vec y)}{(g_k)'_{x_1}(\vec y)}$, we have that 
	
	$$c\nabla g_k(\vec y) =  -\frac{(g_j)'_{x_1}(\vec y)} {(g_k)'_{x_1}(\vec y)}\left(\begin{array}{c}
	(g_k)'_{x_1}(\vec y)\\
	[-2mm]\\
	(g_k)'_{x_2}(\vec y)\\
	\end{array}\right) = - \left(\begin{array}{c}
	(g_j)'_{x_1}(\vec y)\\
	[-2mm]\\
	\frac{(g_k)'_{x_2}(\vec y)(g_j)'_{x_1}(\vec y)}{(g_k)'_{x_1}(\vec y)}\\
	\end{array}\right)$$
	$$ = -\left(\begin{array}{c}
	(g_j)'_{x_1}(\vec y)\\
	[-2mm]\\
	(g_j)'_{x_2}(\vec y)\\
	\end{array}\right)$$
	and
	$$\nabla g_j(\vec y) + c\nabla g_k(\vec y) = \nabla g_j(\vec y) - \nabla g_j(\vec y) = 0.$$
	
	This violates LICQ. %c = \frac{\frac{\partial g'_{i(t_1)}}{\partial x_1}(\vec y)}{\frac{\partial g'_{i(t_2)}}{\partial x_1}(\vec y)}
	
	\item $(\nabla g_j(\vec y))^T\cdot\gamma'(t_2) \neq 0$. This product can be interpreted as the directional derivative of $g_j$ with respect to $\gamma'(t_2)$. Note that $g_j(\vec y) = 0$. Since the directional derivative is non-zero and $\gamma'(t_2)$ locally approximates $\gamma(t)$, then $g_j$ changes sign on $\gamma(t)$ at $t_2$. Then we either have $g_j(\gamma(t_2-\epsilon)) < 0$ and $g_j(\gamma(t_2+\epsilon)) > 0$, or $g_j(\gamma(t_2-\epsilon)) > 0$. In both cases there exist infeasible points on $\gamma(t)$. But since $F$ is a closed set, $\partial F \in F$ and all points $\vec x = \gamma(t), ~t \in [0,T]$ are feasible. Contradiction.
	\end{enumerate}
	\end{proof}
	\qed

	\begin{lemma}\label{grad_prod}
		Consider a boundary point $\vec y = \gamma(t^y)$. If there exist two constraints that are active and non-redundant at $\vec y$, then $\nabla g_{i^-(t^y)}(\vec y) \times \nabla g_{i^+(t^y)}(\vec y) > 0$.
	\end{lemma}
	\begin{proof}
		Consider the vector $\gamma_+'(t^y)$, which is the tangent vector to $g_{i^+(t^y)}$ at point $\vec y$. By definition of $i^+$, constraint $g_{i^+(t^y)}$ is active and non-redundant on $\gamma(t)$ in some right neighborhood of $t^y$. Then the tangent is a feasible direction at $\vec y$ with respect to constraint $g_{i^-(t^y)}(\vec x) \leq 0$. This can be written as:
		
		$$(\nabla g_{i^-(t^y)}(\vec y))^T\cdot \gamma_+'(t^y) \leq 0.$$
		Or, equivalently:
		$$\left(\frac{\partial g'_{i^-(t^y)}}{\partial x_1}(\vec y), \frac{\partial g'_{i^-(t^y)}}{\partial x_2}(\vec y)\right)  \cdot \left(   -\frac{\partial g'_{i^+(t^y)}}{\partial x_2}(\vec y), \frac{g'_{i^+(t^y)}}{\partial x_1}(\vec y)   \right)^T \leq 0 ~\Leftrightarrow$$
		$$-\left(\frac{\partial g'_{i^-(t^y)}}{\partial x_1}(\vec y)   \frac{\partial g'_{i^+(t^y)}}{\partial x_2}(\vec y) - \frac{\partial g'_{i^-(t^y)}}{\partial x_2}(\vec y)   \frac{g'_{i^+(t^y)}}{\partial x_1}(\vec y)\right) \leq 0 ~\Leftrightarrow$$
		$$\left(\frac{\partial g'_{i^-(t^y)}}{\partial x_1}(\vec y)   \frac{\partial g'_{i^+(t^y)}}{\partial x_2}(\vec y) - \frac{\partial g'_{i^-(t^y)}}{\partial x_2}(\vec y)   \frac{g'_{i^+(t^y)}}{\partial x_1}(\vec y)\right) \geq 0 ~\Leftrightarrow$$
		$$\nabla g_{i^-(t^y)}(\vec y) \times \nabla g_{i^+(t^y)}(\vec y) \geq 0.$$
		
		If $\nabla g_{i^-(t^y)}(\vec y) \times \nabla g_{i^+(t^y)}(\vec y) = 0$, then LICQ is violated at point $\vec y$: 
		$$\nabla g_{i^-(t^y)}(\vec y) + c\nabla g_{i^+(t^y)}(\vec y) = 0 \text{ if } c = \left(\frac{\partial g'_{i^-(t^y)}}{\partial x_1}(\vec y)\right)  /  \left(\frac{\partial g'_{i^+(t^y)}}{\partial x_1}(\vec y)\right).$$
		Thus only strict inequality is possible:
		$$\nabla g_{i^-(t^y)}(\vec y) \times \nabla g_{i^+(t^y)}(\vec y) > 0.$$
	\end{proof}
	\qed

%	\section{Boundary invexity}
	
%	Consider the problem (\ref{nlp}). For each non-convex constraint $g_i$ define a problem:
%	
%	\begin{align*}\label{nlpi}
%	& min ~f(x_1,x_2) \tag{NLPi} \\
%	& s.t. ~g_i(x_1,x_2) = 0
%	\end{align*}
	
%	\begin{definition}
%		The optimization problem (\ref{nlp}) is boundary-invex unless for some index $i$ there exists at least one locally optimal solution of (\ref{nlpi}) that belongs to $F$ such that $g_i$ is concave in its neighborhood or a stationary point of (\ref{nlpi}) that is an inflection point of $g_i$ and belongs to $F$.%$f$ has a saddle point on $g_i(x_1,x_2) = 0$.
%	\end{definition}
	
\section{Splitting the space in two}\label{sec:splitting}

\subsection{Behavior of a concave function on a line}\label{section_rays}

First we will prove a general result for one-dimensional real analytic functions.

\begin{lemma}\label{lemma:anfunc}
	Let $f: \rit \rightarrow \rit$ be a real analytic function. If $f$ is constant on some nonempty interval $[a,b]$, then it is identically constant.
\end{lemma}
\begin{proof}
	Suppose that $b$ is the largest number such that $f(x)$ is constant for all $x \in [a,b]$. Since $f$ is real analytic at $b$, at each point $\vec y$ the Taylor series $\sum\limits_{i=0}^{\infty}\frac{f^(n)(\vec y)}{n!}(\vec x-\vec y)$ converges to $f(\vec y)$ \cite{krantz2002primer}. $f$ being constant in some left neighborhood of $b$ implies that left-sided derivatives of any order are equal to $0$ at $b$. Then all coefficients of the Taylor series defining $f$ around $b$ are equal to $0$, so there exists $\epsilon>0$ such that $f(x) = 0 ~\forall x \in (b-\epsilon),(b+\epsilon)$. But then $f$ is constant on $(a,b+\epsilon)$, which is impossible as $b+\epsilon > b$.
\end{proof}
\qed

Let $f: \rit^2 \rightarrow \rit$ be a real analytic concave function. Consider a linear function $l(x_1,x_2) = ax_1 + bx_2 + c$. Let $\vec y$ be a point such that $l(\vec y)=0$. We will define two rays:

\begin{definition}
	$r^d(\vec y)$ is the ray lying on the line $l(\vec x) = 0$ starting at $\vec y$ and pointing in the locally decreasing direction of $f$.
\end{definition}

\begin{definition}
	$r^i(\vec y)$ is the ray lying on the line $l(\vec x) = 0$ starting at $\vec y$ and pointing in the locally increasing direction of $f$.
\end{definition}

%Define two rays lying on this line and starting at $\vec y$: $r^d(\vec y)$ is the ray pointing in the locally decreasing direction of $f$, and $r^i(\vec y)$ is the ray pointing in the locally increasing direction of $f$.

Let $\vec x^{max}_l$ be a point maximizing $f$ subject to $l(\vec x)=0$.

\begin{figure}[h]
	\includegraphics[scale=0.4]{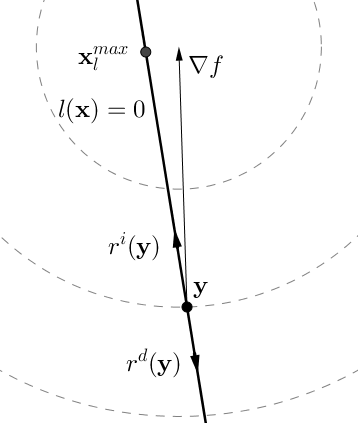}
	\caption{Rays $r^i(\vec y)$ and $r^d(\vec y)$}
\end{figure}

\begin{lemma}\label{lemma_rays1}
	If a concave real analytic function $f(\vec x)$ is not identically constant on $l(\vec x) = 0$ then it is strictly decreasing on $r^d(\vec y)$.
\end{lemma}
\begin{proof}
	Consider two points $\vec x^1, \vec x^2 \in r^d(\vec y)$ such that $||\vec x^2-\vec y|| > ||\vec x^1-\vec y||$. Since $f$ is locally decreasing at $\vec y$ in the direction of $r^d(\vec y)$, $(f'(\vec y))^T\cdot(\vec x^1-\vec y) \leq 0$. By concavity of $f(\vec x)$ we have:
	
	$$f(\vec x^1)-f(\vec y) \leq (f(\vec y))^T\cdot(\vec x^1-\vec y) ~\Rightarrow~ f(\vec x^1)-f(\vec y) \leq 0 ~\Rightarrow~ f(\vec y)-f(\vec x^1) \geq 0.$$
	
	Using the concavity of $f(\vec x)$ again, we get:
	
	$$f(\vec y)-f(\vec x^1) \leq (f(\vec x^1))^T\cdot(\vec y-\vec x^1) ~\Rightarrow~ (f(\vec x^1))^T\cdot(\vec y-\vec x^1) \geq 0$$
	$$\Rightarrow~ (f(\vec x^1))^T\cdot(\vec x^2-\vec x^1) \leq 0.$$
	
	Repeating the same reasoning for $\vec x^1$ and $\vec x^2$ as for $\vec y$ and $\vec x^1$, we can show that $f(\vec x^2) \leq f(\vec x^1)$. 
	
	Since $f(x_1,x_2)$ is real analytic, so is $f(x_1,-\frac{ax_1+c}{b})$, which is the function of one variable $x_1$ and represents the behavior of $f$ on $l(\vec x) = 0$. Since $f(x_1,-\frac{ax_1+c}{b})$ is not identically constant, by Lemma \ref{lemma:anfunc} no interval exists where it is constant. Then strict inequality holds: $f(\vec x^2) < f(\vec x^1)$.
\end{proof}
\qed

%\begin{lemma}\label{lemma_rays2}
%	$f(\vec x)$ is decreasing on $r^i(\vec y)$ for all $\vec x \text{ such that } ||\vec x-\vec y|| > ||\vec x^{max}_l-\vec y||$.
%\end{lemma}
%\begin{proof}
%	Consider two points $\vec x^1, \vec x^2 \in r^i(\vec y)$ such that $||\vec x^2-\vec y|| > ||\vec x^1-\vec y|| > ||\vec x^{max}_l-\vec y||$. Since $\vec x^{max}_l$ maximises $f(\vec x)$ on the line $l(\vec x)=0$, $(f(\vec x^{max}))^T(\vec x-\vec x^{max}_l)=0$. The remaining proof is similar to the proof of Lemma \ref{lemma_rays1} with $\vec x^{max}_l$ instead of $\vec y$.
%\end{proof}
%\qed

	\subsection{Boundary optimality on a half-plane}

	Let $\hat{\vec x} = \gamma(\hat t)$ be a point on the boundary of $F$. In this section we will assume that for the parametrization $\gamma(t)$ defined in Section \ref{sec:bnd}, $f(\gamma(t))$ is non-increasing as a function of $t$ on some interval $[\hat t,\hat t + \epsilon]$, where $\epsilon > 0$. Otherwise, similar results can be proven for the reverse direction parametrization $\gamma^r(t)$.
	
	\begin{definition}
		\cite{mendelson1990introduction} A path in $\rit^n$ is a continuous function mapping every point in the unit interval $[0,1]$ to a point in $\rit^n$:
		$$\rho : [0,1] \rightarrow \rit^n$$
	\end{definition}

	Consider a function $l: \rit^2 \rightarrow \rit$ such that $l(\hat {\vec x}) = 0$. %and $\hat {\vec x}$ is a KKT point of the (\ref{nlp}) problem with an additional constraint $l(\vec x) \leq 0$.
	Let $t^1 > \hat t$ be a parameter value corresponding to the point where $\gamma(t)$ first crosses the line $l(\vec x) = 0$ after $\hat t$:
	
	$$t^1=\begin{cases}
	\min \{t > \hat{t} ~|~ l(\gamma(t)) = 0\} \text{ if such } t \text { exist}\\
	\infty \text{ otherwise }
	\end{cases}$$
	
	$t^1$ exists if $F$ is bounded.
	
	Define the optimization problem
	
	\begin{align*}\label{nlpl}
	\text{max } &f(x_1,x_2) \\
	\text{ s.t. } &g_{i}(x_1,x_2) \leq 0 \text{ } \forall i = 1..m \tag{NLP$_l$} \\
	&l(x_1,x_2) \leq 0 \\
	&(x_1,x_2) \in \rit^{2}.
	\end{align*}
	
	\begin{figure}[h]
		\includegraphics[scale=0.3]{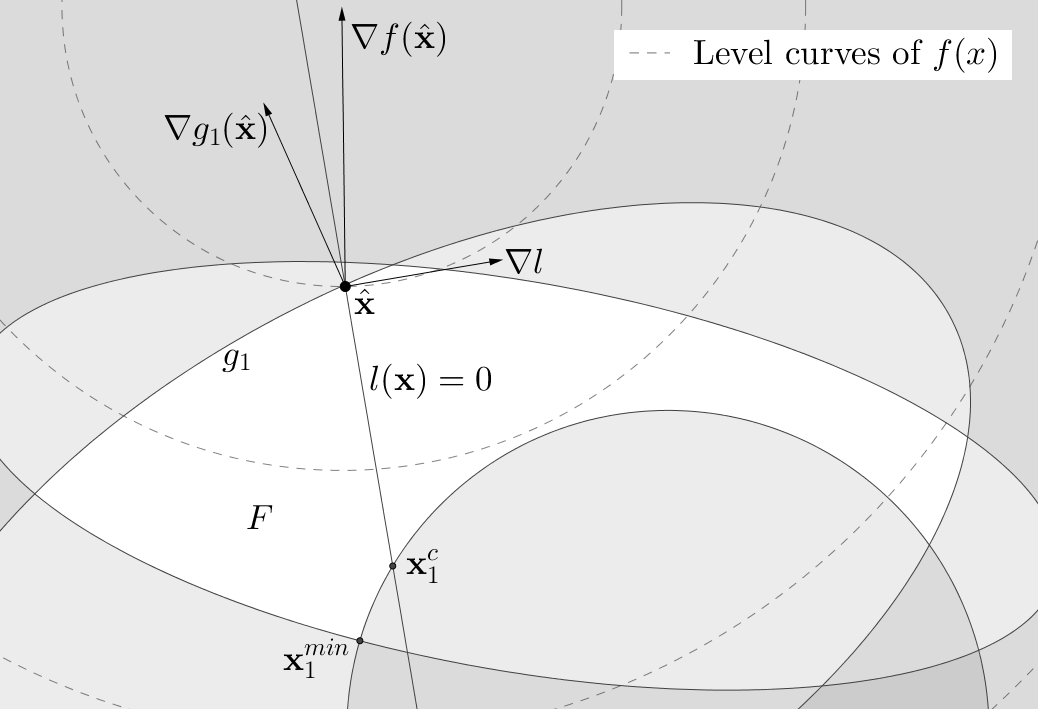}
		\caption{An example problem for Lemma \ref{nextCross}}
	\end{figure}
	
	\begin{lemma}\label{nextCross}
		Given $\gamma(t)$, a parametrization of $\partial F$ in (\ref{nlp}) and given a linear function $l(\vec x)$, if (\ref{nlp}) is boundary-invex and $\hat {\vec x}$ is a KKT point of (\ref{nlpl}), then $f(\gamma(t)) \leq f(\gamma(\hat t)) ~\forall t \in [\hat t, t^1]$.%Consider a function $l: \rit^2 \rightarrow \rit$ such that $l(\hat {\vec x}) = 0$ and $\hat {\vec x}$ is a KKT point of the (\ref{nlp}) problem with an additional constraint $l(\vec x) \leq 0$. %non-redundant constraint $g_1$ such that:
		
		%Let $\hat t$ be a parameter value such that $\gamma(\hat t)=\hat {\vec x}$.
		
		%If to follow $\partial F$ from $\hat {\vec x}$ in the direction of local non-increase of $f$ until the first point where $\partial F$ crosses the line $l(\vec x) = 0$, then $f(\vec x) \leq f(\hat {\vec x})$ for all points on this segment of the boundary curve.
		
		%Suppose that $f(\gamma(t))$ is locally decreasing as a function of $t$ at $\hat t$. Then $f(\gamma(t)) \leq f(\gamma(\hat t)$ for all values of $t$ between $\hat t$ and the value corresponding to the point where $\gamma(t)$ crosses the line $l(\vec x)=0$.
		
		%For the case where $f(\gamma^r(t))$ is locally decreasing as a function of $t$ at $\hat t$, a similar claim holds with $\gamma^r(t)$ instead of $\gamma(t)$.
		
		\end{lemma}
		\begin{proof}

		Let $t^1_{min}$ denote the parameter value corresponding to the point where $f(\gamma(t))$ starts increasing as a function of $t$:
		
		$$\begin{cases}
		t^1_{min} > \hat{t} \\
		(\nabla f(\gamma(t)) \times \nabla g_{i(t)}(\gamma(t))) \geq 0 ~\forall t \in (\hat{t}, t^1_{min})\\
		(\nabla f(\gamma(t)) \times \nabla g_{i(t)}(\gamma(t))) < 0 ~\forall t \in (t^1_{min}, t^1_{min} + \epsilon) \text{ for some } \epsilon > 0

		%\infty \text{ if } \nabla f \times \nabla g_{i(t)} \geq 0 ~\forall t > \hat{t}
		\end{cases}$$
		
		$t^1_{min}$ exists since $F$ is bounded. Let $\vec x_{min}^1 = \gamma(t^1_{min})$.
		
		If $t^1_{min} = t^1$, then for all $\hat t<t\leq t^1$ the inequality $f(\gamma(t)) \leq f(\gamma(\hat t))$ is satisfied and the statement of the lemma holds. Now suppose that $t^1_{min} < t^1$. 
		
%		\textbf{Definition and properties of $S_1$}
		
		Consider the set 
		
		$$L_1 = \left\{\vec x ~\left |~ \begin{cases}
		l(\vec x) \leq 0\\
		f(\vec x) \geq f(\vec x_{min}^1) \end{cases} \right. \right\}$$
		and the curve $\gamma_1(t) = \gamma(t), ~t \in [\hat t,t^1_{min})$. $F$ is connected, $\gamma_1(t)$ is piecewise-continuous, and $\gamma_1(\hat t) = \hat {\vec x}$ lies on the line $l(\vec x) = 0$ and $\gamma_1(t^1_{min})$ lies on the  curve $f(\vec x) = f(\vec x_{min}^1)$, and these are the only points of intersection of the curve and the boundary of $L_1$. Thus $\gamma_1(t)$ is dividing $L_1$ into two connected sets. We will denote the set where all points in the neighborhood of $\gamma_1(t)$ are feasible as $S_1$.
		
		We know that, by definition of $S_1$, all points on its boundary belong to one of the following sets:
		
		\begin{enumerate}
		\item The level curve $f(\vec x) = f(\vec x_{min}^1)$. By definition of $\vec x_{min}^1$ we have that $f(\vec x_{min}^1) \leq f(\hat {\vec x})$.
		
		\item The curve $\gamma_1(t)$. By definition of $\gamma_1(t)$ and $t_{min}^1$, $f(\vec x)\leq f(\hat {\vec x}) ~\forall \vec x \in \gamma_1(t)$.
			
		\item The line $l(\vec x)=0$. Since $\hat {\vec x}$ is a KKT point and thus, by Lemma \ref{lemma_kkt_locmax}, a local maximum, only the direction of local decrease of $f$ on the line is locally feasible. Together with the fact that $\hat {\vec x}$ is the only point where $\gamma_1(t)$ crosses the line, we have that points $\vec x$ in $S_1$ for which $l(\vec x)=0$ lie on the ray $r^d(\hat {\vec x})$ and, by Lemma \ref{lemma_rays1}, satisfy $f(\vec x) \leq f(\hat {\vec x})$.
	    \end{enumerate}
	    
	    Thus $f(\vec x) \leq f(\hat {\vec x}) ~\forall \vec x \in \partial S_1$. By Lemma \ref{lemma_kkt_locmax}, $\hat {\vec x}$ is a local maximum in $S_1$ and thus, by Lemma \ref{lemma_bnd}, $f(\vec x) \leq f(\hat {\vec x}) ~\forall \vec x \in S_1$.

		\textbf{The points following $\gamma(t^1_{min})$ are in $S_1$}
		
		We will say that a path $\rho$ starting at some point $\vec x^s \in \gamma_1(t)$ is $S_1$-feasible if $\vec x \in \rho \implies \vec x \in S_1$.
		
		The definition of $S_1$ implies that for all constraints $g_i$ that are active on $\gamma_1(t)$, $g_i(\vec x) < 0$ for all $\vec x$ on $\rho$ in some neighborhood of $\vec x^s$ excluding $\vec x^s$ itself.
				
%		\begin{enumerate}
%			\item for all constraints $g_i$ that are active on $\gamma_1(t)$ we will require that $g_i(\vec x) < 0$ for all $\vec x$ on $\rho$ in some neighborhood of $\vec x^s$ excluding $\vec x^s$ itself,
%			\item $\vec x \in L_1$ for all $\vec x$ on $\rho$,
%			\item $\rho$ doesn't cross $\gamma_1(t)$.
%		\end{enumerate}}
		
		%The definition of $S_1$ implies that every point $\vec x \in S_1$ can be reached by such a path.
		
		Consider a neighborhood $N(\vec x^1_{min})$ such that only constraints $g_{i^-(t^1_{min})}$ and $g_{i^+(t^1_{min})}$ are non-redundant in it.
		
		Let $t^- = t^1_{min} - \epsilon$ and $t^+ = t^1_{min} + \epsilon$ for some $\epsilon > 0$ and let:
		
		$$\vec x^- = \gamma(t^-), ~\vec x^+ = \gamma(t^+),$$ 
		$$\phi^- = g_{i^-(t_{min}^1)}, ~\phi^+ = g_{i^+(t_{min}^1)}.$$ 
		
		We will show that there exists an $\epsilon_0$ such that for all $\epsilon < \epsilon_0$ the segment connecting $\gamma(t^-)$ and $\gamma(t^+)$ satisfies the conditions defined for the path $\rho$.
		
		Consider two cases:
		
		\begin{enumerate}
			
			\item One constraint is active at $\vec x_{min}^1$.
			
			Define $\phi = \phi^- = \phi^+$.
			
			In this case $\vec x^1_{min}$ is a local minimum of $f$ on $\phi(\vec x) = 0$. Then $\phi$ is either concave or convex in some neighborhood $N(\vec x^1_{min})$. If $\phi$ is concave in $N(\vec x^1_{min})$, then $\vec x^1_{min}$ violates boundary-invexity of (\ref{nlp}). Indeed, this point is a KKT point for (\ref{nlpi}) with a negative KKT multiplier and not a local maximum for (\ref{nlp}).
			
			Then $\phi$ can only be convex in $N(\vec x^1_{min})$.
			
			Since $\vec x^+$ is feasible and belongs to the neighborhood of $\vec x^1_{min}$, then $\phi(\vec x) \leq 0 ~\forall \vec x \in \overline{\vec x^- \vec x^+}$ and $\phi(\vec x) < 0$ for all $\vec x$ on this segment excluding $\vec x^-$. Hence $\vec x^+ \in S_1$.
			
%			Since the problem (\ref{nlp}) is boundary-invex, we have that: $\nabla \phi(\vec x^1_{min}) = -c\nabla f(\vec x^1_{min}), ~c > 0$. Then if $t(\vec x) = 0$ is the tangent line to the level curve $f(\vec x) = f(\vec x^1_{min})$ at $\vec x^1_{min}$, it is also the tangent line to $\phi(\vec x) = 0$.\\
%			
%			Suppose that $\nabla f(\vec x^1_{min}) = \nabla t$. The definition of $\vec x^1_{min}$ implies that in its neighborhood $t(\vec x) \geq 0$ on the curve $\phi(\vec x) = 0$. Since $\nabla \phi(\vec x^1_{min}) = -c\nabla f(\vec x^1_{min})$, the function $\phi$ decreases in the direction $\nabla f(\vec x^1_{min})$. Then in the neighborhood of $\vec x^1_{min}$ the set defined by $\phi(\vec x) \leq 0$ fully lies on one side of the tangent line and thus $\phi$ is locally convex.\\ 
						
			\item Two constraints are active at $\vec x_{min}^1$.
			
			By Lemma \ref{grad_prod}, $\nabla \phi^- (\vec x_{min}^1) \times \nabla \phi^+(\vec x^1_{min}) > 0$. By definitions of the two-dimensional cross product, this is equivalent to:
			
			$$(\nabla \phi^-(\vec x^1_{min}))^T\cdot \gamma'_+(t^1_{min}) < 0$$
			
			This product can be interpreted as the directional derivative of $\phi^-$ with respect to the vector $\gamma'_+$. Observe that $\gamma'_+(t^1_{min})$ shows how $\vec x$ behaves on $\gamma(t)$ when small changes to $t$ are made. Therefore, the above inequality implies that there exists $\epsilon_0$ such that for any $\epsilon < \epsilon_0$ the following holds:
			
			$$(\nabla \phi^-(\vec x^1_{min}))^T\cdot (\vec x^+ - \vec x^1_{min}) < 0$$
			
			Since all constraints are twice continuously differentiable, $\nabla \phi^-(\vec x) (\vec x^+ - \vec x)$ is a differentiable function of $\vec x$. Thus there exists a neighborhood $N(\vec x^1_{min})$ where this function stays negative. We can choose $\epsilon_0$ such that $\vec x^- \in N(\vec x^1_{min}) ~\forall \epsilon < 0$ and:
			
			$$(\nabla \phi^-(\vec x^-))^T\cdot (\vec x^+ - \vec x^-) < 0$$
			
			There exists $\epsilon_0$ such that $\phi^-(\vec x) \leq 0 ~\forall \vec x \in \overline{\vec x^- \vec x^+}$ if $\epsilon < \epsilon_0$. Thus the segment $\overline{\vec x^- \vec x^+}$ is an $S_1$-feasible path.
			
%			\begin{enumerate}
%				\item There exists a neighborhood of $\vec x^-$ where $\phi(\vec x) < 0$ on $\overline{\vec x^- \vec x^+}$.
%				\item Since both $\vec x^-$ and $\vec x^+$ lie within the convex level set $f(\vec x) \geq f(\vec x^1_{min})$, the whole segment $\overline{\vec x^- \vec x^+}$ lies within this level set too.
%				\item Since $\phi^-$ is the only constraint that is active on $\gamma_1(t)$ and non-redundant in the $\epsilon_0$-neighborhood of $\vec x^1_{min}$, the segment doesn't cross $\gamma_1(t)$.
%			\end{enumerate}

		\end{enumerate}
		
		\textbf{Exiting $S_1$}
		
		By Lemma \ref{lemma_no_self_intersect}, $\gamma(t)$ cannot intersect itself and therefore cannot cross $\gamma_1(t)$. Consequently, there are only two ways of exiting $S_1$:
		
		\begin{enumerate}
			\item Crossing the level curve. Then $f$ is decreasing on $\gamma(t)$ at the intersection point. Let the next point where $f(\gamma(t))$ starts increasing again be denoted as $t^2_{min}$ and define $\gamma_2(t) = \gamma(t), ~t \in [\hat t, t^2_{min}]$. This curve has the same properties as $\gamma_1(t)$:
			
			\begin{enumerate}\item $f(\vec x) \leq f(\hat{\vec x})$ for all $\vec x$ on $\gamma_2(t)$ and
				
			\item $\gamma_2(t)$ only crosses the line $l(\vec x) = 0$ at $\hat {\vec x}$ and the level curve $f(\vec x) = f(\vec x^2_{min})$ at $\vec x^2_{min}$, where $\vec x^2_{min} = \gamma(t^2_{min})$.
			\end{enumerate}
			
			Then $S_2$ can be defined similarly to $S_1$ with the new parameters and the same reasoning can be repeated.
			
			\item Cross $l(\vec x)=0$. Then $f(\gamma(t)) \leq f(\gamma(\hat t)) ~\forall t \in [\hat t, t^1]$.
		\end{enumerate}

	\end{proof}
	\qed
	
	\begin{lemma} \label{lemma_tc_decr}
	Consider a point $\hat {\vec x}$ satisfying the conditions of Lemma \ref{nextCross} with $l(\vec x)$ and $\gamma(t)$. Let $\vec x^1 \in r^i(\hat {\vec x})$ be the next point where $\gamma(t)$ crosses the line after $\hat{\vec x}$, then $\vec x^1$ satisfies the conditions of Lemma \ref{nextCross} for $\gamma(t)$ and $-l(\vec x)$.
	\end{lemma}
	\begin{proof}
		Let $t^1$ be defined similarly to Lemma \ref{nextCross} and $\vec x^1 = \gamma(t^1)$.
		
		It follows immediately from the definition of $\vec x^1$ that $l(\vec x^1) = 0$.
		
		First let us prove that $f(\gamma(t))$ is non-increasing as a function of $t$ at $t^1$. Assume the contrary: $f(\gamma(t))$ strictly increases as a function of $t$ at $t^1$. Then there exists a $t^i_{min}$ such that $\hat t < t_{min}^i < t^1$ and $f(\gamma(t))$ is monotone on the $[t^i_{min}, t^1]$ interval.
		
		Then there exists a set $S_i$ and, as proved in the previous lemma, if $\vec x \in r^i(\hat {\vec x})$ then $\vec x \notin S_i$. Then $\gamma(t)$ has to exit $S_i$ at some $t<t^1$. There are two possibilities:
		
		\begin{enumerate}
		\item $\gamma(t)$ crosses $r^d(\hat {\vec x})$. This contradicts with $\gamma(t^1) \in r^i(\hat {\vec x})$, 
		\item $\gamma(t)$ crosses the level curve. Then $f(\gamma(t))$ decreases somewhere between $t_{min}^1$ and $t^1$. This contradicts with $f(\gamma(t))$ being monotonic on $[t^i_{min},t^1]$.
		\end{enumerate}		
		This proves that $f(\gamma(t))$ is non-increasing at $t^1$. 		\\
		Now we shall show that $x^1$ is a local maximizer of $f$ in $F \cap l(\vec x) \geq 0$.		\\
		By Lemma \ref{lemma_cross_prod}, $f(\gamma(t))$ being non-increasing at $t^1$ implies that:
		
		\begin{equation}\label{loc_opt1} \nabla f(\vec x^1) \times \nabla g_{i(t^1)}(\vec x^1) \geq 0.\end{equation}
		
		Since $\gamma(t)$ crosses the line from the $l(\vec x) \leq 0$ half-space into the $l(\vec x) \geq 0$ half-space at $\vec x^1$ , $l(\gamma(t))$ is increasing at $t^1$ and thus, by Lemma \ref{lemma_cross_prod}, we have that $\nabla l \times \nabla g_{i(t^1)}(\vec x^1) \leq 0$ or, equivalently:
		
		\begin{equation}\label{loc_opt2} (-\nabla l) \times \nabla g_{i(t^1)}(\vec x^1) > 0.\end{equation}
		
		Finally, by Lemma \ref{nextCross}, $f(\vec x^1) \leq f(\hat{\vec x})$ and thus $\vec x^1$ belongs to the part of ray $r^i(\hat{\vec x})$ where $f$ is decreasing. If we consider the direction which $r^i(\hat{\vec x})$ points to as the positive direction of moving along the line, then the corresponding gradient is $-\nabla l$. Then Lemma \ref{lemma_cross_prod} implies that
		
		\begin{equation}\label{loc_opt3}\nabla f(\vec x^1) \times (-\nabla l) \geq 0.\end{equation}

		By Lemma \ref{lemma:kkt_prod_signs}, these inequalities imply that $\vec x^1$ is a KKT point in $F \cap \{l(\vec x) \geq 0\}$. Thus the conditions of Lemma \ref{nextCross} are satisfied at $\vec x^1$ for $F \cap \{l(\vec x) \geq 0\}$.
	\end{proof}
	\qed
	
\section{Kuhn-Tucker invexity of boundary-invex problems}\label{sec:main}

%Given two points $\vec x_i, \vec x_j$ such that $l(\vec x_i)=0$ and $l(\vec x_j)=0$, let $S(\vec x_i,\vec x_j)$ denote the set inside the curve which consists of $\gamma(t)$ with $t \in [t_i,t_j]$ and the $\overline{\vec x_i\vec x_j}$ segment of the $l(\vec x)=0$ line.%\vec xi+1

%Let $\vec x^*$ be a point on the boundary of the feasible region of the (\ref{nlp}) problem satisfying the conditions of Lemma \ref{nextCross} for some linear function $l(\vec x)$ and the positive direction parametrisation of the boundary $\gamma(t)$.

\subsection{Sequence of crossing points}

Consider a point $\vec x^*$ which is a local maximum of (\ref{nlp}) and a linear function $l(\vec x)$ such that $f$ is not constant on $l(\vec x) = 0$. Let $\gamma(0) = \vec x^*$.

Given two parameter values $r, s$, let $\hat \gamma(r,s)$ denote the segment of the $\gamma(t)$ curve with $t \in [r,s]$.

Let $\vec x^i$ be the $i^{th}$ point where $\gamma(t)$ crosses $l(\vec x)=0$ and let $t^i$ be a parameter value such that $\vec x^i = \gamma(t^i)$. Since $\gamma(t)$ is a closed curve, $\vec x^i$ exists for each $i \in \nit$ if at least one crossing point exists. %If $\vec x^i$ does not exist, set $t^i$ to $\infty$. Let $t^i$ be the $i^{th}$ point where $\gamma^r$ crosses $l(\vec x)=0$ and $\vec y^i_y = \gamma^r(t^i_y)$.

%From (\ref{g1inL}) we have: 
%
%\begin{gather*}
%(l'_{x_1}g'_{i^+(t_0){x_2}} - l'_{x_2}g'_{i^+(t_0){x_1}})(\hat{\vec x}) > 0  ~\Rightarrow~ (\nabla l \gamma'_+(t_0))(\hat{\vec x}) < 0 ~\Rightarrow  l(\hat{\vec x} + \gamma'_+(t_0)\epsilon) < 0%\\
%%(l'_xg'_{i^{r+}(t^r_m)y} - l'_yg'_{i^{r+}(t^r_m)x})(\hat{\vec x}) > 0  ~\Rightarrow~ -(\nabla l 		\gamma^{r\prime}_+(t^r_m))(\hat{\vec x}) > 0 ~\Rightarrow  l(\hat{\vec x} - \gamma^{r\prime}_+(t^r_m)\epsilon) > 0,
%\end{gather*}
%
%where $\epsilon > 0$. This implies that $l(\gamma(t)) < 0$ for $t^0 < t < t^1$.% and $l(\gamma^r(t)) > 0$ for $0 < t < t^r_1$.
%
%Then for parameters $t^i$ we have:

%The numbering of the crossing points will be chosen so that:
%
%\begin{gather*}
%\begin{cases}
%l(\gamma(t)) \leq 0 ~\forall t \in [t^{i-1},t^i] \\
%l(\gamma(t)) \geq 0 ~\forall t \in [t^i,t^{i+1}]
%\end{cases} \text{if} ~i \in 2\nit{+}1,
%\\
%\begin{cases}
%l(\gamma(t)) \geq 0 ~\forall t \in [t^{i-1},t^i] \\
%l(\gamma(t)) \leq 0 ~\forall t \in [t^i,t^{i+1}]
%\end{cases} \text{if} ~i \in 2\nit,
%\end{gather*}
%
%and at least on one side of $t^i$ the corresponding inequality holds strictly.
%
%Similarly, for $t^i_y$ the following holds:
%
%\begin{gather*}
%\begin{cases}
%l(\gamma^r(t)) \geq 0 ~\forall t \in [t^{i-1}_y,t^i_y] \\
%l(\gamma^r(t)) \leq 0 ~\forall t \in [t^i_y,t^{i+1}_y]
%\end{cases} \text{if} ~i \in 2\nit{+}1,
%\\
%\begin{cases}
%l(\gamma^r(t)) \leq 0 ~\forall t \in [t^{i-1}_y,t^i_y] \\
%l(\gamma^r(t)) \geq 0 ~\forall t \in [t^i_y,t^{i+1}_y]
%\end{cases} \text{if} ~i \in 2\nit.
%\end{gather*}

The numbering of the crossing points will be chosen so that the even indices will correspond to $\gamma(t)$ crossing the line $l(\vec x) = 0$ from $l(\vec x) > 0$ into $l(\vec x) < 0$, and the odd indices will correspond to the opposite direction of crossing.

%Note that this defines the sign of the product $\nabla l \times \nabla g_{i(t^i)}(\vec x^i)$, which is, by Lemma \ref{lemma_cross_prod}, indicating the monotonicity of $l(\vec x)$ for $\vec x$ on $g_{i(t^i)}(\vec x^i) = 0$.

\begin{lemma}\label{lemma:cross_on_rd}
Consider a crossing point $\vec x^i$, $i \in 2\nit$. If $\nabla l \times \nabla f(\vec x^i) \geq 0$, then $\vec x^i$ satisfies Lemma \ref{nextCross} for either $l(\vec x)$ and $\gamma(t)$ or for $-l(\vec x)$ and $\gamma^r(t)$.
\end{lemma}
\begin{proof}
	Since $\gamma(t)$ crosses the line from $l(\vec x) < 0$ into $l(\vec x) > 0$ at $\vec x^i$, we have that $\nabla l \times \nabla g_{i(t^i)}(\vec x^i) < 0$. 
	
	By Lemma \ref{lemma:kkt_prod_signs}, $\vec x^i$ is a KKT point in one of the following sets:
	
	\begin{enumerate}
		\item $F \cap \{l(\vec x) \leq 0\}$ if $\nabla f \times \nabla g_{i(t^i)} \geq 0$. The latter inequality also implies that Lemma \ref{nextCross} is satisfied at $\vec x^i$ for $\gamma(t)$ and $l(\vec x)$ (see the beginning of Section \ref{sec:splitting}).
		
		\item $F \cap \{-l(\vec x) \leq 0\}$ if $\nabla f \times \nabla g_{i(t^i)} \leq 0$.  The latter inequality implies that Lemma \ref{nextCross} is satisfied at $\vec x^i$ for $\gamma^r(t)$ and $-l(\vec x)$.
	\end{enumerate}
\end{proof}
\qed

Let $S(\overline{AB},\overline{BC},...) \subset F$ denote a set with the boundary comprised of some sections of $\partial F$ and segments $\overline{AB}$, $\overline{BC}$, $...$ on the line $l(\vec x = 0)$.

\begin{definition}
	$S(\overline{AB},\overline{BC},...) \subset F$ is a safe set if $f(\vec x) \leq f(\vec x^*) ~\forall \vec x \in S$.
\end{definition}

\begin{figure}[h]
	\includegraphics[scale=0.35]{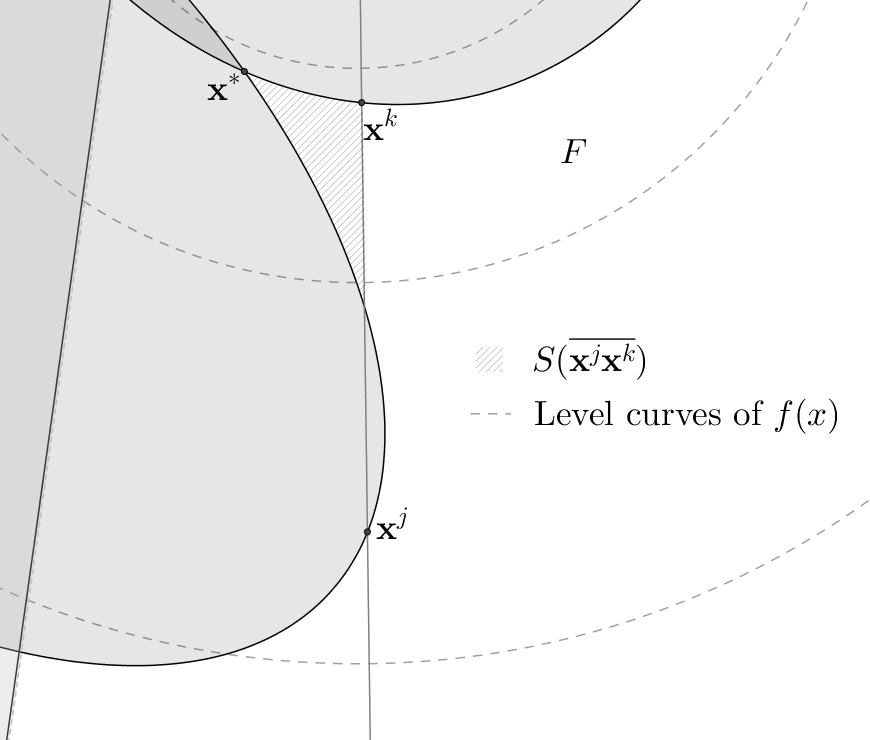}
	\caption{Points $\vec x^*$, $\vec x^j$, $\vec x^k$ and set $S(\overline{\vec x^j\vec x^k})$ satisfying the conditions of Theorem \ref{theorem_sequence}}
	\label{fig:Tstart}
\end{figure}
	
\begin{theorem}\label{theorem_sequence}
	Consider points $\vec x^j$, $\vec x^k \in F$ such that:
	
	$\vec x^k, ~k \in 2\nit$, satisfies Lemma \ref{nextCross} for $\gamma^r$ and $-l$, $f(\vec x^k) \leq f(\vec x^*)$;
	
	$\vec x^j \in r^d(\vec x^k), ~j \in 2\nit$, satisfies Lemma \ref{nextCross} for $\gamma$ and $l$;
	
	$f(\gamma(t)) < f(\vec x^*) ~\forall t \in [0,t^j]$;
	
	if $\vec x^j \neq \vec x^k$ and $\gamma(t)$ crosses $\overline{\vec x^j\vec x^k}$ from $l(\vec x)>0$ into $l(\vec x)<0$, it enters a safe set $S(\overline{\vec x^j\vec x^k})$ with the boundary consisting of $\overline{\vec x^j\vec x^k}$ and $\hat \gamma(t^k,t^{j-1})$.
	
%	$\vec x^j$ satisfies Lemma \ref{nextCross} for $\gamma(t)$ and $l(\vec x)$;
%	
%	$\vec x^k \in \overline{\vec x^j\vec x^*}$ and satisfies Lemma \ref{nextCross} for $\gamma^r(t)$ and $-l(\vec x)$, $f(\vec x^k) \leq f(\vec x^*)$;
%	
%	$f(\gamma(t)) < f(\vec x^*) ~\forall t \in [0,t^j]$;
%	
%	if $j \neq 0$ and $\gamma$ crosses $\overline{\vec x^{j}\vec x^k}$ from $l(\vec x)>0$ into $l(\vec x)<0$, it enters a safe set $S(\overline{\vec x^{j}\vec x^k})$.
	
	Then $x^*$ is the global optimum of (\ref{nlp}).%$f(\vec x) \leq f(\vec x^*) ~\forall \vec x = \gamma(t), ~t \in [0,T]$.
\end{theorem}
\begin{proof}
	The conditions on $\vec x^k$ imply that $\nabla f(\vec x^k) \times \nabla l \leq 0$. By Lemma \ref{lemma_rays1}, $f$ is monotonically decreasing on the whole ray $r^d(\vec x^k)$ and thus $\nabla f(\vec x) \times \nabla l \leq 0 ~\forall \vec x \in r^d(\vec x^k)$. Then points $\vec x^i \in r^d(\vec x^k), ~i \in 2\nit$, satisfy conditions of Lemma~\ref{lemma:cross_on_rd}.

	Let us consider the following cases:
	
	\begin{enumerate}
		\item $\vec x^{j+1} \in r^d(\vec x^k)$.\\
		
		Let $S(\overline{\vec x^j\vec x^{j+1}})$ be the set with the boundary composed of $\hat \gamma(t^j,t^{j+1})$ and the segment $\overline{\vec x^j\vec x^{j+1}}$. By Lemma \ref{nextCross}, $f(\vec x) \leq f(\vec x^j) ~\forall \vec x \in \hat \gamma(t^j,t^{j+1})$. Since the segment $\overline{\vec x^j\vec x^{j+1}}$ is part of the $r^d(\vec x^j)$ ray, then by Lemma \ref{lemma_rays1}, $f$ is decreasing on this segment from $x^j$ in the direction of $x^{j+1}$ and thus $f(\vec x) \leq f(\vec x^j) ~\forall \vec x \in \overline{\vec x^j\vec x^{j+1}}$. Since $\vec x^j$ satisfies the conditions of Lemma \ref{nextCross}, it is a local maximum in $S(\overline{\vec x^j\vec x^{j+1}})$. Then, by Lemma \ref{lemma_bnd}, $f(\vec x) \leq f(\vec x^j) \leq f(\vec x^*) ~\forall \vec x \in S(\overline{\vec x^j\vec x^{j+1}})$. Thus $S(\overline{\vec x^j\vec x^{j+1}})$ is a safe set.\\
		
		By Lemma \ref{lemma_no_self_intersect}, $\gamma(t)$ cannot exit $S(\overline{\vec x^j\vec x^{j+1}})$ by crossing itself. Then the only way to exit $S(\overline{\vec x^j\vec x^{j+1}})$ is to cross the $\overline{\vec x_j\vec x_{j+1}}$ line segment again.\\
		
		Let $S(\overline{\vec x^{j}\vec x^k}, \overline{\vec x^j\vec x^{j+1}}) = S(\overline{\vec x^{j}\vec x^k}) \cup S(\overline{\vec x^j\vec x^{j+1}})$. Since it is a union of two safe sets, $S(\overline{\vec x^{j}\vec x^k}, \overline{\vec x^j\vec x^{j+1}})$ is a safe set.\\
		
		If $\vec x^{j+2} = \vec x^k$, then, by Lemma \ref{nextCross} applied to $x^k$, $-l$ and $\gamma^r$, $f(\vec x) \leq f(\vec x^k) \leq f(\vec x^*) ~\forall \vec x \in \hat \gamma(t^{j+1},t^{j+2})$. Since the conditions of the theorem imply that $f(\gamma(t)) \leq f(\vec x^*) ~\forall t \in [k,T]\cup[0,j]$, we have that $f(\gamma(t)) \leq f(\vec x^*) ~\forall t \in [0,T]$.\\
		
		We will consider the following cases that depend on the position of $\vec x^{j+2}$ on $l(\vec x) = 0$:\\
		
%		This case can be applied whenever the following holds: 1) $\vec x_{j-1}$ is a crossing point from $l(\vec x)<0$ to $l(\vec x)>0$ 2) $\vec x^r_m$ is between $\vec x^*$ and $\vec x_{j-1}$ on the $l(\vec x)=0$ line 3) $\vec x^r_m$ satisfies Lemma \ref{nextCross} for $\gamma^r$ 4) crossing back between $\vec x^r_m$ and $\vec x_{j-1}$ means entering a region where $f(\vec x) \leq f(\vec x^*)$.\\
		
		\begin{enumerate}
			\item $\vec x^{j+2} \in \overline{\vec x^k\vec x^{j+1}}$.\\
			
			$\gamma(t)$ enters $S(\overline{\vec x^{j}\vec x^k}, \overline{\vec x^j\vec x^{j+1}})$ at $\vec x^{j+2}$. Since it is a safe set, $f$ cannot reach values larger than $f(\vec x^*)$ unless $\vec x^{j+3}$ exists. Repeat case (1) with $\vec x^{j+3}$ instead of $\vec x^{j+1}$.
			\\
			
			\item $\vec x^{j+2} \in r^d(\vec x^{j+1})$.\\%\in r^d(\vec x^k)$ and $||\vec x^{j+2}-\vec x^*||>||\vec x^{j+1}-\vec x^*||$. \\
			
			$j{+}2 \in 2\nit$. Since $\vec x^{j+2} \in r^d(\vec x^{j+1}) \in r^d(\vec x^k)$ and, by Lemma \ref{lemma_rays1}, a concave function is always decreasing in the direction of local decrease from a given point, the monotonicity of $f$ on $l(\vec x) = 0$ at $\vec x^{j+2}$ is similar to that at $\vec x^k$. This implies that $sign(\nabla l \times \nabla f(\vec x^{j+2})) = sign(\nabla l \times \nabla f(\vec x^k)) \geq 0$. Then, by Lemma \ref{lemma:cross_on_rd}, one of the following is true at $\vec x^{j+2}$:
			
			\begin{enumerate}
			\item $f(\gamma(t))$ is increasing at $t^{j+2}$. Then at this point Lemma \ref{nextCross} can be applied for $\gamma^r$ and $-l$ to show that $f(\gamma(t)) \leq f(\vec x^{j+2}) ~\forall t \in (t^{j+1},t^{j+2})$. But by Lemma \ref{lemma_rays1}, $f$ is non-increasing on $r^d(\vec x^k)$ and $f(\vec x^{j+2}) < f(\vec x^{j+1})$. This contradicts with $f(\vec x^{j+2}) \leq f(\vec x^{j+1})$.

%			By definition of $\vec x^{j+1}$ and $\vec x^{j+2}$ we have:
%			
%			\begin{gather}
%			l(\gamma(t)) > 0 ~\forall t \in (t^{j+1}, t^{j+2}) \label{beforet2} \\
%			l(\gamma(t)) < 0 ~\forall t \in (t^{j+2}, t^{j+3}), \label{aftert2}
%			\end{gather}
%			
%			or $l(\gamma(t))$ decreases at $t^{j+2}$ as a function of $t$. This implies that for the vector $\gamma'_+(t^{j+2})$, which is the directed tangent to the $\gamma(t)$ curve at $t^{j+2}$, the following holds: $\gamma'_+(t^{j+2}) \nabla l < 0 ~\Rightarrow~ \nabla g(\vec x^{j+2}) \times \nabla l < 0$, and thus condition (\ref{g1inL}) of Lemma \ref{nextCross} is satisfied. Since $\vec x^{j+2}$ is a point of intersection of $\gamma(t)$ with the $l(\vec x)=0$ line, we have that $l(\vec x^{j+2}) = 0$ and $g_{i(t^{j+2})}(\vec x) = 0$ and condition (\ref{actAtPhat}) is satisfied. By Lemma \ref{lemma_rays1}, $f(\vec x)$ is decreasing on $r^d(\vec x^k)$, which implies that $\nabla f(\vec x^{j+2}) \times \nabla l < 0$ and (\ref{monLine}) holds.\\
			
			\item $f(\gamma(t))$ is decreasing at $t^{j+2}$. Then Lemma \ref{nextCross} is satisfied at $\vec x^{j+2}$ for $\gamma$ and $l$. Then $\vec x^{j+2}, \vec x^k$ and the $l(\vec x)=0$ line satisfy the conditions of this theorem and the reasoning can be repeated from the start.
			\end{enumerate}

			\item $\vec x^{j+2} \in r^i(\vec x^k)$.\\
			
			\begin{figure}[h]
				\includegraphics[scale=0.4]{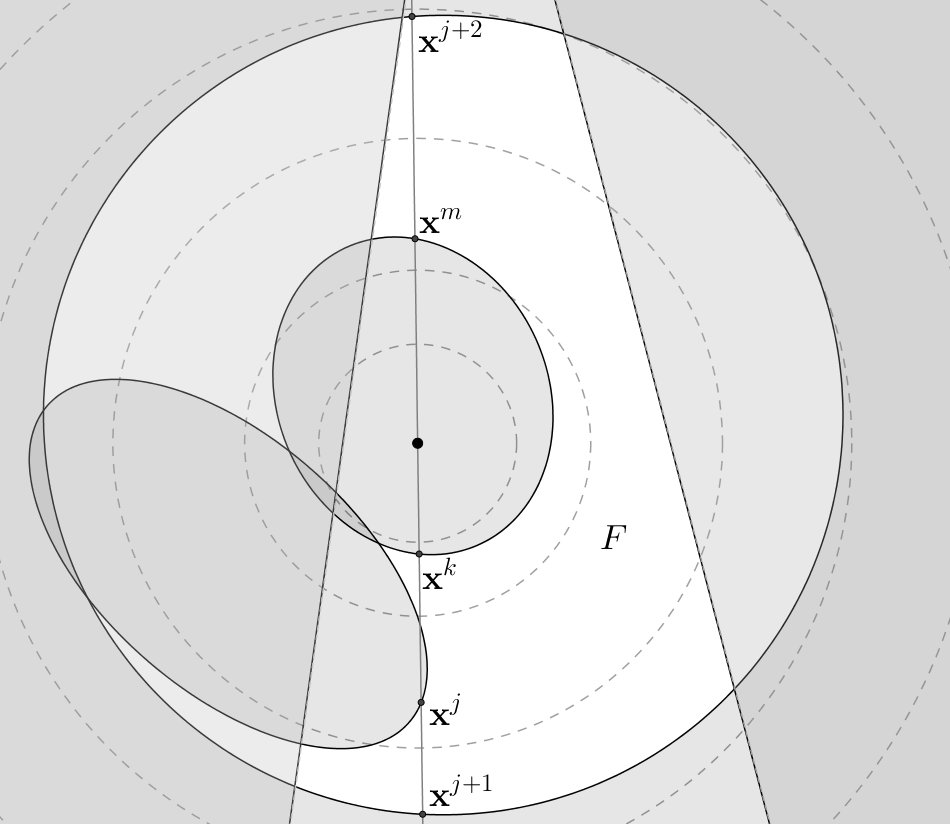}
				\caption{Case (1c): $\vec x^{j+1} \in r^d(\vec x^k)$}
				\label{fig_inv_case1}
			\end{figure}
			
			%Let $\vec x^m$, $m \in 2\nit+1$, be the last crossing point of $\gamma(t)$ and $r^i(\vec x^k)$ with $t < t^k$. $\vec x^m$ exists since $\gamma(t)$ crosses $r^i(\vec x^k)$ at least once. $\vec x^{m+1}$ lies on $r^d(\vec x^k)$.\\
			
%			Since $k, j+2 \in 2\nit$, there exist some crossing points with odd indices on the line segment $\overline{x^kx^{j+2}}$. Let $\vec x^m$, $m \in 2\nit+1$, be the last crossing point of $\gamma(t)$ and $r^i(\vec x^k)$ with $t < t^k$. If $x^m \in r^i(\vec x^{j+2})$, then $\gamma(t)$ cannot reach $\vec x^k$ without either crossing $\hat \gamma(\vec x^{j+1},\vec x^{j+2})$ (which is not permitted) or crossing the line again. But in the latter case, it is not the last crossing point on $r^i(\vec x^k)$ with $t < t^k$, which contradicts with the definition of $\vec x^m$. Then $\vec x^m \in \overline{\vec x^k\vec x^{j+2}}$.
			
			Let $S(\overline{\vec x^{j+1}\vec x^{j+2}})$ be the set with the boundary composed of $\overline{\vec x^{j+1}\vec x^{j+2}}$ and $\gamma(t^{j+1},t^{j+2})$. Let $S(\overline{\vec x^{j+2}\vec x^k}) = S(\overline{\vec x^{j+1}\vec x^{j+2}})$ $\cup S(\overline{\vec x^k\vec x^j})$ $\cup S(\overline{\vec x^j\vec x^{j+1}})$.\\
			
			At $\vec x^{j+2}$ $\gamma(t)$ leaves $S(\overline{\vec x^{j+2}\vec x^k})$. But $\vec x^k$ belongs to the boundary of $S(\overline{\vec x^{j+2}\vec x^k})$ and $\gamma(t)$ approaches $\vec x^k$ from the interior of this set. This implies that at some point $\gamma(t)$ enters $S(\overline{\vec x^{j+2}\vec x^k})$. Let $\vec x^m$ denote the last such point on $\gamma(t)$ before $x^k$. Then the next crossing point $\vec x^{m+1}$ can only belong to $\overline{\vec x^k\vec x^{j+1}}$.\\
			
			\textbf{$f(\gamma(t))$ is increasing at $t^m$}\\	
			
%			First we will show that $f(\gamma(t))$ increases at $t = t^{m+1}$. If $\vec x^{m+1} = \vec x^k$, this follows from the conditions of the theorem. Now suppose that $\vec x^{m+1} \neq \vec x^k$. Assume that $f(\gamma(t))$ decreases at $t = t^{m+1}$. We will show that in this case no $i$ exists such that $\vec x^{m+i} = \vec x^k$.\\
%			
%			Lemma \ref{nextCross} can be applied at $t^{m+1}$ for parametrisation $\gamma(t)$ and $l(\vec x)$ to show that $f(\gamma(t^{m+2})) \leq f(\gamma(t^{m+1}))$. Combined with the fact that $\vec x^{m+2} \in r^d(\vec x^k)$ and $f(\vec x)$ is decreasing on $r^d(\vec x^k)$, this implies that $\vec x^{m+2} \in r^d(\vec x^{m+1})$ and $\vec x^{m+2} \neq \vec x^k$.\\
%			
%			$\vec x^{m+3}$ can't belong to $\overline{\vec x^m\vec x^{m+1}}$ since this would imply $\gamma(t)$ intersecting itself. If $\vec x^{m+3} \in \overline{\vec x^{m+1}\vec x^{m+2}}$, then $\gamma(t)$ enters a set where $f(\vec x) \leq f(\vec x^{m+1}) < \vec x^k$.\\
%			
%			Then the only case which could allow reaching $\vec x^k$ is $\vec x^{m+3} \in r^d(\vec x^{m+2})$. But then, similarly to case (1b), we have two successive crossing points on $\gamma(t)$ and $\vec x^{m+3}$ satisfies Lemma \ref{nextCross} for $\gamma(t)$ and $l(\vec x)$. Applying the same reasoning for $\vec x^{m+3}$ instead of $\vec x^{m+1}$ proves that the sequence of crossing points $\vec x^{m+i}$, $i=1,2,3,...$ will never reach $\vec x^k$. Contradiction with $\vec x^k \in \partial F$.\\
%			
%			This proves that $f(\gamma(t))$ increases at $t = t^{m+1}$.
			
			Consider the point $\vec x^{k-1}$. If $\vec x^{k-1} = \vec x^m$, then the proof is done.\\
			
			Now suppose that $\vec x^{k-1} \neq \vec x^m$. The definition of $\vec x^m$ implies that $\vec x^{k-1} \in \overline{\vec x^k\vec x^{j+1}}$. The points following $\vec x^{k-1}$ on $\gamma^r(t)$ belong to one of the sets $S(\overline{\vec x^k\vec x^j})$, $S(\overline{\vec x^j\vec x^{j+1}})$. Thus $f(\gamma(t)) \leq f(\vec x^*) ~\forall t \in [k-1,k]$ and $\vec x^{k-2}$ exists and belongs to $\overline{x^kx^{j+1}}$. Consider the following cases:\\
			
			\begin{enumerate}
				\item $\vec x^{k-2} \in \overline{\vec x^k\vec x^{k-1}}$. Then $\gamma^r(t)$ enters the set $S(\overline{\vec x^k\vec x^{k-1}})$ and $\vec x^{k-3} \neq x^m$. Repeat case (1ci) with $\vec x^{k-3}$ instead of $\vec x^{k-1}$.\\
				
				\item $\vec x^{k-2} \in \overline{\vec x^{k-1}\vec x^{j-1}}$. Then, similarly to case (1b), $\vec x^{k-2}$ satisfies the conditions of Lemma \ref{nextCross} for $-l$ and $\gamma^r$. Thus $f(\gamma(t)) \leq f(\vec x^*) ~\forall t \in [t^{k-2},t^{k-3}]$. If $\vec x^{k-3} = \vec x^m$, by Lemma \ref{lemma_tc_decr} $\vec x^m$ satisfies the conditions of Lemma \ref{nextCross} for $l$ and $\gamma$. Otherwise repeat (1ci) and (1cii) with $\vec x^{k-2}$ instead of $\vec x^k$ and $\vec x^{k-3}$ instead of $\vec x^{k-1}$.\\
			\end{enumerate}
			
			We have proven that $f(\gamma(t)) \leq f(\vec x^*) ~\forall t \in [t^k,t^m]$ and $\vec x^m$ satisfies the conditions of Lemma \ref{nextCross} for $l$ and $\gamma$.\\

			\textbf{Starting a new iteration}\\

			Consider the set $S(\overline{\vec x^{j+2}\vec x^m})$ that contains the section of the $\gamma(t)$ curve from $t^m$ to $t^{j+2}$. If this set is disconnected, then there exist points $\vec x \in S(\overline{\vec x^{j+2}\vec x^m})$ that cannot be connected to the segment $\overline{\vec x^{j+2}\vec x^m}$ by a continuous path that belongs to this set. But since every feasible path from $S(\overline{\vec x^{j+2}\vec x^m})$ to $F \setminus S(\overline{\vec x^{j+2}\vec x^m})$ crosses $\overline{\vec x^{j+2}\vec x^m}$, this implies that there is no feasible path from $\vec x$ to points in $F \setminus S(\overline{\vec x^{j+2}\vec x^m})$ and thus $F$ is disconnected. This contradicts with the theorem assumptions. Hence $S(\overline{\vec x^{j+2}\vec x^m})$ is a connected set.\\
			
			We have shown that $f(\vec x) \leq f(\vec x^*)$ for all $\vec x$ on this curve. $\vec x^*$ is a local maximum in $S(\overline{\vec x^{j+2}\vec x^m})$. Then $f(\vec x) \leq f(\vec x^*) ~\forall \vec x \in S(\overline{\vec x^{j+2}\vec x^m})$.\\
			
			Case (1) of this theorem can be repeated with $\vec x^{j+2}$, $S(\overline{\vec x^{j+2}\vec x^m})$, $-l$ and $\vec x^m$ instead of $\vec x^{j+1}$, $S(\overline{\vec x^{j}\vec x^k}, \overline{\vec x^j\vec x^{j+1}})$, $l$ and $\vec x^k$.
			\\
			
		\end{enumerate}
		
		\begin{figure}[h]
			\includegraphics[scale=0.4]{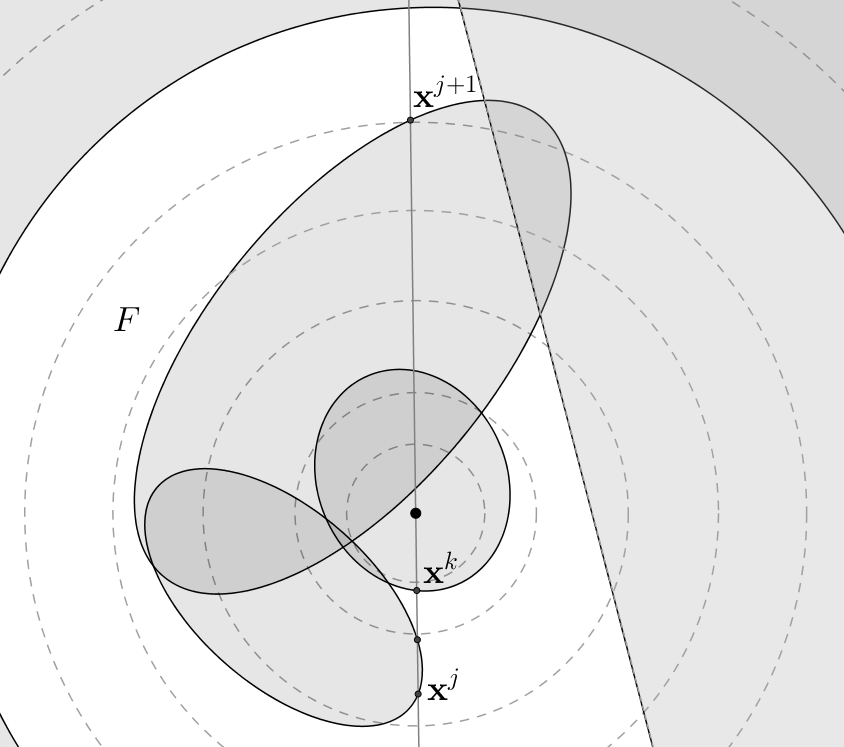}
			\caption{Case 2: $\vec x^{j+1} \in r^i(\vec x^k)$}
			\label{fig:case2}
		\end{figure}
		
		\item $\vec x^{j+1} \in r^i(\vec x^k)$. By Lemma \ref{lemma_tc_decr}, $f(\gamma(t))$ is decreasing at $t^{j+1}$ and $\vec x^{j+1}$ satisfies the conditions of Lemma \ref{nextCross}. Then $f(\vec x) \leq f(\vec x^{j+1})$ until the next crossing point $\vec x^{j+2}$.
		\begin{enumerate}
			\item $\vec x^{j+2} \in \overline{\vec x^k\vec x^{j+1}}$.%\overline{\vec x^*\vec x^{j+1}}$.
			
			The assumptions of this theorem imply that $f(\vec x^k) > f(\vec x_j) \geq f(\vec x^{j+1})$. This means that $\vec x^k$ belongs to the increasing section of the ray $r^i(\vec x^{j+1})$ and $f(\vec x) > f(\vec x^{j+1}) ~\forall \vec x \in \overline{\vec x^k\vec x^{j+1}}$. Then $f(\vec x^{j+2}) > f(\vec x^{j+1})$. Contradiction with $f(\vec x^{j+2}) \leq f(\vec x^{j+1})$.
			\\
			
			\item $\vec x^{j+2} \in r^d(\vec x^k)$.
			
%			While $l(\gamma(t)) < 0$, $f$ on $\gamma$ can't exceed $f(\vec x^*)$ without crossing the segment of $\gamma$ from $\vec x_j$ to $\vec x_{j+1}$. Similarly, if $\gamma$ crosses into $l(\vec x)>0$ outside of the $\overline{\vec x_{j+1}\vec x_{j+2}}$ segment of the line, $f$ on $\gamma$ can't exceed $f(\vec x^*)$ without crossing the segment of $\gamma$ from $\vec x_{j+1}$ to $\vec x_{j+2}$. Then the only way to reach points where $f(\vec x)>f(\vec x^*)$ is to cross the $\overline{\vec x_{j+1}\vec x_{j+2}}$ segment from $l(\vec x)<0$ at some point $\vec x_i$. In this case, (1) can be repeated: $B_i$ is composed of $B_j$ and the set whose boundary is defined by the $\gamma(t)$ curve with $t \in [t_{j+1},t_i]$ and the $\overline{\vec x_{j+1}\vec x_i}$ segment of the line.

			By applying Lemma \ref{lemma_tc_decr} to $\vec x^{j+2}$ we can show that this point satisfies the conditions of Lemma \ref{nextCross} for $\gamma$ and $l$. Then $\vec x^{j+2}$ has the same properties as $\vec x^j$. Repeat everything with same $\vec x^k$ and $\vec x^{j+2}$ instead of $\vec x^j$.
			\\
			
			\item $\vec x^{j+2} \in r^d(\vec x^{j+1})$.% and $||\vec x^{j+2}-\vec x^*||>||\vec x^{j+1}-\vec x^*||$.
			
			\begin{enumerate}
				\item $\vec x^{j+3} \in r^d(\vec x^{j+2})$. Repeat (2) with $\vec x^{j+3}$ instead of $\vec x^{j+1}$.
				
				\item $\vec x^{j+3} \in \overline{\vec x^{k}\vec x^{j+2}}$. From $\vec x^{j+2}$ $\gamma(t)$ cannot reach the line segment $\overline{\vec x^k\vec x_{j+1}}$ without crossing $\hat \gamma(t^j,t^{j+1})$. Then $\vec x^{j+2} \in \overline{\vec x^{j+1}\vec x^{j+2}}$ and $\gamma(t)$ enters a safe set $S(\overline{\vec x^{j+1}\vec x^{j+2}})$. Then $\vec x^{j+4} \in \overline{\vec x^{j+1}\vec x^{j+2}}$. Repeat (2c) with $\vec x^{j+4}$ instead of $\vec x^{j+2}$.
				
				\item $\vec x^{j+3} \in r^d(\vec x^k)$. Repeat (1) with $\vec x^{j+3}$ instead of $\vec x^{j+1}$.
			\end{enumerate}

%			All points in $l(\vec x) < 0$ such that $f(\vec x) > f(\vec x^*)$ are cut off from $\vec x_{j+2}$ by the $\vec x^*, \vec x_1$ part of the $\gamma$ curve. $f$ might be increased above $f(\vec x^*)$ only after crossing the $l(\vec x) = 0$ line at some point $\vec x_{j+3}$.
			
			%Repeat (2) with $\vec x^{j+3}$ instead of $\vec x^{j+1}$.

		\end{enumerate}
	\end{enumerate}
	
	We have proven that $f(\gamma(t)) \leq f(\vec x^*) ~\forall t\in[0,T]$. By Lemma \ref{lemma_bnd}, together with the fact that $\vec x^*$ is a local maximum this implies that $\vec x^*$ is the global maximum of (\ref{nlp}).
	
\end{proof}
\qed

\subsection{The main theorem}

\begin{theorem}
	If (\ref{nlp}) is boundary-invex, then it is KT-invex.
\end{theorem}
\begin{proof}
	Let $\vec x^*$ be a KKT point. If $\vec x^*$ lies in the interior of $F$, then, by concavity of $f$, it is the global unconstrained maximum of $f$ and thus the global maximum for (\ref{nlp}).
	
	Now suppose that $\vec x^* \in \partial F$. Let $\gamma(0)=\gamma(T)=\vec x^*$ in the parametrization of $\partial F$. By Lemma \ref{lemma_bnd}, it is enough to consider only the values on the boundary. We need to prove that there exists a line $l(\vec x) = 0$ such that the conditions of Theorem \ref{theorem_sequence} are satisfied for the point $\vec x^* = \vec x_j = \vec x_k$.
	
	If $\nabla g_1(\vec x^*) \times \nabla g_2(\vec x^*) = 0$, then $\nabla g_1(\vec x^*) = c\nabla g_2(\vec x^*)$, and LICQ is violated. Now suppose that $\nabla g_1(\vec x^*) \times \nabla g_2(\vec x^*) \neq 0$. Since $\vec u \times \vec v = -(\vec v \times \vec u)$, we can assume w.l.o.g. that $\nabla g_1(\vec x^*) \times \nabla g_2(\vec x^*) < 0$. Then, by Lemma \ref{lemma:kkt_prod_signs}, the following holds:
	
	\begin{gather}
	\nabla f(\vec x^*) \times \nabla g_1(\vec x^*) \geq 0 \label{cp1} \\
	\nabla f(\vec x^*) \times \nabla g_2(\vec x^*) \leq 0 \label{cp2}
	\end{gather}

	Consider a linear function $l(x_1,x_2)$ such that $l(\vec x^*) = 0$ and $\nabla l = -\nabla g_1(\vec x^*) + \nabla g_2(\vec x^*)$. By (\ref{cp1}), (\ref{cp2}) we have:
	
	$$\nabla f(\vec x^*) \times \nabla l = -\nabla f(\vec x^*) \times \nabla g_1(\vec x^*) + \nabla f(\vec x^*) \times \nabla g_2(\vec x^*) < 0$$
	and
	\begin{gather*}
	\nabla l \times \nabla g_1(\vec x^*) = -\nabla g_1(\vec x^*) \times \nabla g_1(\vec x^*) + \nabla g_2(\vec x^*) \times \nabla g_1(\vec x^*) = \\ \nabla g_2(\vec x^*) \times \nabla g_1(\vec x^*) > 0, \\
	\nabla l \times \nabla g_2(\vec x^*) = -\nabla g_1(\vec x^*) \times \nabla g_2(\vec x^*) + \nabla g_2(\vec x^*) \times \nabla g_2(\vec x^*) = \\ -\nabla g_1(\vec x^*) \times \nabla g_2(\vec x^*) > 0.
	\end{gather*}
	
	%Thus at $\vec x^*$ conditions (\ref{monLine}, \ref{g1inL}) of Lemma \ref{nextCross} are satisfied for both $g_1$ and $g_2$. (\ref{actAtPhat}) is satisfied by definition of $l$ and choice of $g_1$, $g_2$. By (\ref{cp1}, \ref{cp2}), conditions of Lemma \ref{nextCross} are satisfied for both $\gamma$ and $\gamma^r$.
	
	By Lemma \ref{lemma:kkt_prod_signs}, these inequalities imply that $\vec x^*$ is a KKT point in both $F \cap \{l(\vec x) \leq 0\}$ and	$F \cap \{l(\vec x) \geq 0\}$. $\nabla f(\vec x^*) \times \nabla l < 0$ also implies that $f$ is non-constant on $l(\vec x) = 0$. Theorem \ref{theorem_sequence} can then be applied to show that $\vec x^*$ is the global maximum of (\ref{nlp}).
\end{proof}
\qed

\section{Application}\label{sec:app}

	\subsection*{Notations}
	\begin{tabular}{l}		
   {$\bm i$} - imaginary number constant\\
   {$S = p+ \bm iq$} - Electric power\\
   {$V = v \angle \theta$}  - Voltage\\
   {$\bm Y = \bm g + \bm i \bm b$}  - Line admittance\\
   {$W \!= w^R \!+ \bm i w^I $}  - Product of two voltages\\
   {$\bm s^u$} - Line apparent power thermal limit\\
   {$\theta_{ij}$} - Phase angle difference (i.e., $\theta_i - \theta_j$)\\
  % [{$\delta$}] - Phase angle difference offset
   {$\bm S^d = \bm p^d+ \bm i \bm q^d$} - Power demand\\
   {$S^g = p^g+ \bm iq^g$} - Power generation\\
   {$\bm c_0, \bm c_1, \bm c_2$} - Generation cost coefficients \\
    {$\Re(\cdot)$} - Real part of a complex number\\
    {$\Im(\cdot)$} - Imaginary part of a complex number\\
    {$(\cdot)^*$} - Conjugate of a complex number\\
    {$|\cdot|$} - Magnitude of a complex number, $l^2$-norm\\
   % [{$\|\cdot\|_{2}$}] - $l^2$-norm, Euclidean norm
   %
    {$\bm x^l, \bm x^u$} - Lower and upper bounds of $x$\\
    \end{tabular}

    \subsection{The Power Flow Equations}
In Power Systems, the Alternating Current (AC) power flow equations link the complex quantities of voltage $V$, power $S$, and admittance
$Y$, using Ohm's and Kirchhoff's Current Laws.
They can be written as,
\begin{subequations}
\begin{align}
& S^g_i - {\bm S^d_i} = \sum_{\substack{(i,j)\in E}} S_{ij} + \sum_{\substack{(j,i)\in E}} S_{ij} \;\; \forall i\in N \label{eq:ac_kcl} \\ 
& S_{ij} = \bm Y^*_{ij} V_i V^*_i - \bm Y^*_{ij} V_i V^*_j \;\; (i,j),(j,i) \in E \label{eq:ac_ohm}
\end{align}
\end{subequations}
\noindent
A detailed derivation of these equations can be found in \cite{qc_opf_tps}.
The non-convex nonlinear equations \eqref{eq:ac_kcl}--\eqref{eq:ac_ohm} form the core building block of many power network optimization applications.
These equations are usually augmented with side constraints such as,
%
%\begin{subequations}
\begin{align}
%& \bm {p^{gl}}_i \leq p^g_i \leq \bm {p^{gu}}_i \;\; \forall i \in N \\
%& \bm {q^{gl}}_i \leq q^g_i \leq \bm {q^{gu}}_i \;\; \forall i \in N
& \Re(\bm {S^{gl}}_i) \leq \Re(S^g_i) \leq \Re(\bm {S^{gu}}_i) \;\; \forall i \in N \label{eq:p_gen}  \\
& \Im(\bm {S^{gl}}_i) \leq \Im(S^g_i) \leq \Im(\bm {S^{gu}}_i) \;\; \forall i \in N \label{eq:q_gen} \\
& ( \bm {v^l}_{i} )^2 \leq |V_i|^2 \leq ( \bm {v^u}_{i} )^2 \;\; \forall i \in N \label{eq:v_mag_sqr} \\
& |S_{ij}| \leq \bm {s^u}_{ij} \;\; \forall (i,j),(j,i) \in E \label{eq:thermal_limit} \\
& \tan(\bm {\theta^l}_{ij}) \Re(V_iV^*_j) \! \leq \!  \Im(V_iV^*_j) \! \leq \! \tan(\bm {\theta^u}_{ij}) \Re(V_iV^*_j) \;  \forall (i,j) \! \in \!  E.  \label{eq:w_pad} 
\end{align}
%\end{subequations}
Constraints \eqref{eq:p_gen}--\eqref{eq:q_gen} set limits on the real and reactive generator capabilities, respectively.
Constraints \eqref{eq:v_mag_sqr} limit the magnitudes of bus voltages.
Constraints \eqref{eq:thermal_limit} limit the power flow on the lines and constraints \eqref{eq:w_pad} limit the difference of the phase angles (i.e., $\theta_i, \theta_j$) between the lines' buses.  A detailed derivation and further explanation of these operational side constraints can be found in \cite{qc_opf_tps}.
\subsection{Optimal Power Flow}
\begin{model}[!h]
\caption{ The AC Optimal Power Flow Problem (AC-OPF).}
\label{model:ac_opf_w}
\begin{subequations}
\vspace{-0.2cm}
\begin{align}
& \mbox{\bf variables: } \nonumber \\
& S^g_i \in ( \bm {S^{gl}}_i, \bm {S^{gu}}_i) \;\; \forall i\in N \nonumber \\
& V_i \in ( \bm {V^l}_i, \bm {V^u}_i ) \;\; \forall i\in N \nonumber \\
& W_{ij} \in ( \bm {W^l}_{ij}, \bm {W^u}_{ij} )  \;\; \forall i \in N,  \forall j \in N \nonumber \\
& S_{ij} \in (\bm {S^{l}}_{ij},\bm {S^{u}}_{ij})\;\; \forall (i,j),(j,i) \in E \nonumber \\
& \mbox{\bf minimize: } \nonumber \\ 
& \sum_{i \in N} \bm c_{2i} (\Re(S^g_i))^2 + \bm c_{1i}\Re(S^g_i) + \bm c_{0i} \label{w_obj} \\
& \mbox{\bf subject to: } \nonumber  \\
& \angle V_{\bm r} = 0 \\
&  W_{ij} = V_iV_j^* \;\; \forall (i,j)\in E \label{w_2} \\
& S^g_i - {\bm S^d_i} = \sum_{\substack{(i,j)\in E}} S_{ij} + \sum_{\substack{(j,i)\in E}} S_{ij}\;\; \forall i\in N \label{w_5} \\ 
& S_{ij} = \bm Y^*_{ij} W_{ii} - \bm Y^*_{ij} W_{ij} \;\; \forall (i,j)\in E \label{w_6} \\
& S_{ji} = \bm Y^*_{ij} W_{jj} - \bm Y^*_{ij} W_{ij}^* \;\; \forall (i,j)\in E \label{w_7} \\
& |S_{ij}| \leq (\bm {s^u}_{ij}) \;\; \forall (i,j),(j,i) \in E \label{w_8} \\
& \tan(\bm {\theta^l}_{ij}) \Re(W_{ij}) \! \leq \! \Im(W_{ij}) \! \leq \! \tan(\bm {\theta^u}_{ij}) \Re(W_{ij}) \;  \forall (i,j) \! \in \!  E \label{w_9} \nonumber \\
\end{align}
\end{subequations}
\end{model}

The AC Optimal Power Flow Problem (ACOPF) combines the above power flow equations, side constraints, and a convex objective function as described in Model \ref{model:ac_opf_w}.
This formulation utilizes a voltage product factorization $V_i V_j^* = W_{ij} \;\; \forall (i,j)\in E$.
Model \ref{model:ac_opf_w} is a non-convex nonlinear optimization problem, which has been shown to be NP-Hard in general \cite{verma2009power,ACSTAR2015}.
In real-world deployments, the AC-OPF problem is solved with numerical methods such as \cite{744492,744495}, which are not guaranteed to converge to a feasible point and provide only stationary points (e.g., saddle points or local minimas) when convergence is achieved.

In the following section we look at a family of ACOPF problems with two degrees of freedom and show that they are boundary-invex under mild assumptions on the variables' bounds. Namely, we will enforce that 
$$-\frac{\pi}{6} \leq \bm \theta^l < \bm \theta^u \leq \frac{\pi}{6} \text{ and } 0.95 \leq \bm {v^l}_{i} < \bm {v^u}_{i} \leq 1.05.$$

%\begin{align}
%& p_{ij} = gw_i - gw^R_{ij} - bw^I\\
%& q_{ij} = -bw_i + bw^R_{ij} - gw^I\\
%& kcl\\
%& w_2 = \frac{(w^R)^2+(w^I)^2}{w}\\
%& p_{12}^2 + q_{12}^2 \leq t_{ij}\\
%& \underbar{w}_i \leq w_i \leq \bar{w}_i \\
%& w^R\tan(\theta^l) \leq w^I \leq w^R\tan(\theta^u)
%\end{align}
%
%Since this is a minimization problem, (\ref{nlpi}) are maximization problems in this case and their KKT points with positive multipliers might violate boundary-invexity.

\subsection{Boundary-invex ACOPF}
\begin{model}[h!]
	\caption{ AC-OPF for 1-line networks.}
	\label{model:ac_opf_1}
	\begin{subequations}
		\vspace{-0.2cm}
		\begin{align}
		& \mbox{\bf variables: } \nonumber \\
		& w^R,~ w^I \nonumber \\
		& \mbox{\bf minimize: } \nonumber \\ 
		& \bm c_1(\bm g\bm w - \bm gw^R - \bm bw^I) + \bm c_2(\frac{\bm g}{\bm w}((w^R)^2+(w^I)^2) - \bm gw^R + \bm bw^I) \\
		& \mbox{\bf subject to: } \nonumber  \\
		& (\bm g\bm w - \bm gw^R - \bm bw^I)^2 + (-\bm b\bm w + \bm bw^R - \bm gw^I)^2 \leq \bm s^u\\
		& (\frac{\bm g}{\bm w}((w^R)^2+(w^I)^2) - \bm gw^R + \bm bw^I)^2 \nonumber \\
		& + (-\frac{\bm b}{\bm w}((w^R)^2+(w^I)^2) + \bm bw^R + \bm gw^I)^2 \leq \bm s^u \label{therm_lim2}\\
		& (\bm{p}_1^g)^l - \bm p_1^d \leq \bm g\bm w - \bm gw^R - \bm bw^I \leq (\bm{p}_1^g)^u - \bm p_1^d\\
		& (\bm{q}_1^g)^l - \bm q_1^d \leq -\bm b\bm w + \bm bw^R - \bm gw^I \leq (\bm{q}_1^g)^u - \bm q_1^d\\
		& (\bm{p}_2^g)^l - \bm p_2^d \leq \frac{\bm g}{\bm w}((w^R)^2+(w^I)^2) - \bm gw^R + \bm bw^I \leq (\bm{p}_2^g)^u - \bm p_2^d \label{pbound} \\
		& (\bm{q}_2^g)^l - \bm q_2^d \leq -\frac{\bm b}{\bm w}((w^R)^2+(w^I)^2) + \bm bw^R + \bm gw^I \leq (\bm{q}_2^g)^u - \bm q_2^d \label{qbound}\\
		& (\bm {v^l}_{2})^2 \leq \frac{(w^R)^2+(w^I)^2}{\bm w} \leq (\bm {v^u}_{2})^2 \label{wbound}\\
		& w^R\tan(\bm\theta^l) \leq w^I \leq w^R\tan(\bm\theta^u) \label{tbound}
		\end{align}
	\end{subequations}
\end{model}

\begin{figure}[h]
	\includegraphics[scale=0.45]{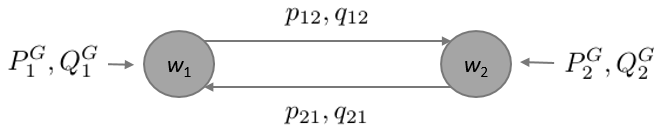}
	\caption{1-line network\label{fig:2d}}
\end{figure}

Consider a 2-bus network with one line and two generators as depicted in Figure \ref{fig:2d}.
We assume the voltage magnitude to be fixed at node $1$.
%\begin{equation*}\label{eq:lift}
%\text{Let }W_{ij} = {w}^R_{ij} + \bm i {w}^I_{ij} = V_iV_j^*
%\end{equation*}
For clarity purposes we will adopt the following notations: $\bm w = \bm w_1$, $w^R = w^R_{12}$, $w^I = w^I_{12}$ and $\bm s^u = \bm s^u_{12}$.
The real-number formulation of Model \ref{model:ac_opf_w} is given in Model \ref{model:ac_opf_1}.

%\begin{subequations}\label{OPF1}
	%min ~$c_1(gw - gw^R - bw^I) + c_2(\frac{g}{w}((w^R)^2+(w^I)^2) - gw^R + bw^I)$\\
%	s.t.
%	\begin{align}
%	& (gw - gw^R - bw^I)^2 + (-bw + bw^R - gw^I)^2 \leq t\\
%	& (\frac{g}{w}((w^R)^2+(w^I)^2) - gw^R + bw^I)^2 \nonumber \\
%	& + (-\frac{b}{w}((w^R)^2+(w^I)^2) + bw^R + gw^I)^2 \leq t \label{therm_lim2}\\
%	& \underbar{p}_1^g - \bm p_1^d \leq gw - gw^R - bw^I \leq \bar{p}_1^g - \bm p_1^d\\
%	& \underbar{q}_1^g - \bm q_1^d \leq -bw + bw^R - gw^I \leq \bar{q}_1^g - \bm q_1^d\\
%	& \underbar{p}_2^g - \bm p_2^d \leq \frac{g}{w}((w^R)^2+(w^I)^2) - gw^R + bw^I \leq \bar{p}_2^g - \bm p_2^d \label{pbound} \\
%	& \underbar{q}_2^g - \bm q_2^d \leq -\frac{b}{w}((w^R)^2+(w^I)^2) + bw^R + gw^I \leq \bar{q}_2^g - \bm q_2^d \label{qbound}\\
%	& (\bm {v^l}_{2})^2 \leq \frac{(w^R)^2+(w^I)^2}{w} \leq \bar{w}_2 \label{wbound}\\
%	& w^R\tan(\theta^l) \leq w^I \leq w^R\tan(\theta^u) \label{tbound}
%	\end{align}
%\end{subequations}

%where $\bm c_i$ denotes the generation cost at node $i$, $\bm g$ and $\bm b$ represent the conductance and susceptance of the line, $\bm p_i^d$, $\bm q_i^d$ stand for the active and reactive loads at node $i$ and $\underbar p^g_i, \underbar q^g_i$ and $\bm{\bar p}^g_i, \bm{\bar q}^g_i$ represent the lower and upper bounds on active and reactive power generation.

This model has four non-convex constraints, namely, (\ref{therm_lim2}), (\ref{pbound}), (\ref{qbound}) and (\ref{wbound}).

\begin{figure}[h]
	\includegraphics[scale=0.4]{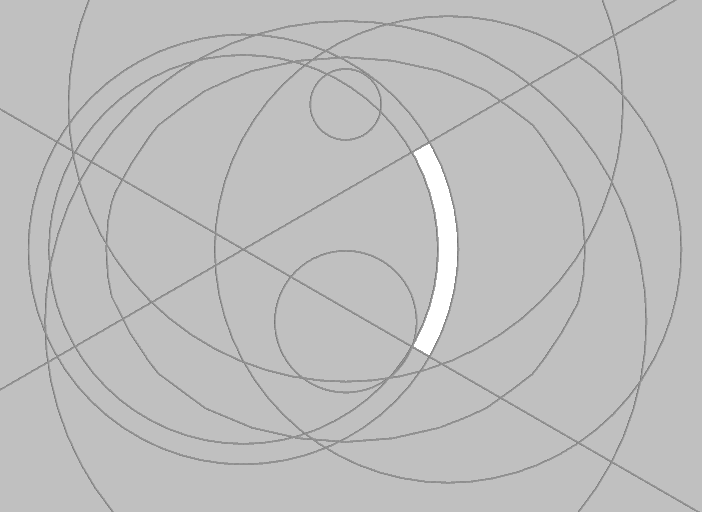}
	\caption{Feasible set of Model 2 (feasible region in white)}
\end{figure}

\paragraph{Minimal feasible $w^R$}

\begin{lemma}\label{lemma:minfeas}
	 If $w^R \leq 0.77\bm w$, then $(w^R,w^I)$ is infeasible for Model \ref{model:ac_opf_1}.
\end{lemma} 
\begin{proof}
Consider the lower bound on $w_2$ and voltage angle bounds. No feasible points exist where:

$$w^R\tan(\bm\theta^l) \leq w^I \leq w^R\tan(\bm\theta^u) \implies (w^R)^2 + (w^I)^2 < (\bm {v^l}_{2})^2\bm w $$

If $w^R \geq (\bm {v^l}_{2})^2\bm w$, the latter is always false. Suppose that $w^R < (\bm {v^l}_{2})^2\bm w$. Consider the $w^I \geq 0$ half-space. The lower angle bound is redundant here, and the remaining two inequalities can be written as:

$$w^I \leq w^R\tan(\bm\theta^u)$$
$$w^I < \sqrt{(\bm {v^l}_{2})^2\bm w - (w^R)^2}$$

The implication holds if the second inequality is dominated by the first:

$$w^R\tan(\bm\theta^u) < \sqrt{(\bm {v^l}_{2})^2\bm w - (w^R)^2} \Leftrightarrow$$
$$(w^R)^2\tan^2(\bm\theta^u) < (\bm {v^l}_{2})^2\bm w - (w^R)^2$$

It can be seen that only points with non-negative $w^R$ can satisfy constraint (\ref{tbound}). Then the above is equivalent to:

$$w^R < \sqrt{\frac{(\bm {v^l}_{2})^2\bm w}{\tan^2(\bm\theta^u)+1}}.$$

Since $\bm\theta^u \leq \frac{\pi}{6}$ and $(\bm {v^l}_{2})^2 \geq (0.95)^2$, we have that:

$$\sqrt{\frac{(\bm {v^l}_{2})^2\bm w}{\tan^2(\bm\theta^u)+1}} \geq \sqrt{\frac{(0.95)^2\bm w}{\frac{1}{3}+1}} \geq 0.82\sqrt{\bm w}$$

All $w^R < 0.82\sqrt{\bm w}$ are guaranteed to be infeasible. Since $\bm w \leq 1.1025$, it can be shown that $\sqrt{\bm w} \geq 0.95\bm w$ and thus all $w^R < 0.77\bm w$ are infeasible.
\end{proof}
\qed

\paragraph{$w_2$ lower bound}

Consider constraint (\ref{wbound}). Let $g_1(w^R,w^I) = (\bm {v^l}_{})^2 - \frac{(w^R)^2+(w^I)^2}{\bm w}$. 

\begin{lemma}\label{lemma:wbnd}
	KKT points of problem (\ref{nlpi}), $i = 1$ do not violate the boundary-invexity for Model \ref{model:ac_opf_1}.
\end{lemma}
\begin{proof}
(\ref{nlpi}) takes the following form for $i=1$:

\begin{align*}
& \max ~c_1(\bm g\bm w-\bm gw^R-\bm bw^I) + c_2(\frac{\bm g}{\bm w}((w^R)^2+(w^I)^2) - \bm gw^R + \bm bw^I) \\
& s.t. ~(w^R)^2 + (w^R)^2 = (\bm {v^l}_{})^2\bm w
\end{align*}

which can be rewritten as

\begin{align*}
& \max ~\bm c_1(\bm g\bm w-\bm gw^R-\bm bw^I) + \bm c_2(\frac{\bm g}{\bm w}(\bm {v^l}_{})^2\bm w - \bm gw^R + \bm bw^I) \\
& s.t. ~(w^R)^2 + (w^I)^2 = (\bm {v^l}_{})^2\bm w
\end{align*}

The KKT conditions for this problem are:

\begin{align}
& -\bm g(\bm c_1+\bm c_2) = 2\lambda \hat w^R \label{g1kkt1} \\
& \bm b(\bm c_2-\bm c_1) = 2\lambda \hat w^I \\
& ((\hat w^R)^2 + (\hat w^I)^2 = (\bm {v^l}_{})^2\bm w
\end{align}

The solution of this system can violate boundary-invexity only if $\lambda > 0$. It can be seen from (\ref{g1kkt1}) that $\hat w^R < 0$ if $\lambda > 0$. But since by Lemma \ref{lemma:minfeas} all points $(w^R,w^I)$ such that $w^R < 0.77\bm w$ are infeasible, $(\hat w^R,\hat w^I)$ is infeasible.
\end{proof}
\qed

\paragraph{$p_{21}$ lower bound}\label{lemma:pbnd}

Consider constraint (\ref{pbound}). Let $g_2 = (\bm{p}_2^g)^l - \bm p_2^d - \frac{\bm g}{\bm w}((w^R)^2+(w^I)^2) + \bm gw^R - \bm bw^I$.

\begin{lemma}
	KKT points of problem (\ref{nlpi}), $i = 2$ do not violate the boundary-invexity for Model \ref{model:ac_opf_1}.
\end{lemma}
\begin{proof}
(\ref{nlpi}) takes the following form for $i=2$:

\begin{align*}
& \max ~\bm c_1(\bm g\bm w-\bm gw^R-\bm bw^I) + \bm c_2(\frac{\bm g}{\bm w}((w^R)^2+(w^I)^2) - \bm gw^R + \bm bw^I) \\
& s.t. ~\frac{\bm g}{\bm w}((w^R)^2 + (w^I)^2) = (\bm{p}_2^g)^l - \bm p_2^d+\bm gw^R-\bm bw^I
\end{align*}

which can be rewritten as

\begin{align*}
& \max ~\bm c_1(\bm g\bm w-\bm gw^R-\bm bw^I) + \bm c_2((\bm{p}_2^g)^l - \bm p_2^d) \\
& s.t. ~\frac{\bm g}{\bm w}((w^R)^2 + (w^I)^2) = (\bm{p}_2^g)^l - \bm p_2^d+\bm gw^R-\bm bw^I
\end{align*}

The KKT conditions for this problem are:

\begin{align*}
\bm c_1\bm g &= -\lambda(\frac{2\bm g\hat w^R}{\bm w} - \bm g) \\
\bm c_1\bm b &= -\lambda(\frac{2\bm g\hat w^I}{\bm w} + \bm b) \\
~\frac{\bm g}{\bm w}((\hat w^R)^2 + (\hat w^I)^2) &= (\bm{p}_2^g)^l - \bm p_2^d+\bm g\hat w^R-\bm b\hat w^I
\end{align*}
and the first equation implies that
$$\hat w^R = -\frac{\bm c_1\bm w}{2\lambda} + \frac{\bm w}{2}.$$

The solution of this system can violate boundary-invexity only if $\lambda > 0$. Then $\hat w^R < \frac{\bm w}{2}$. But since by Lemma \ref{lemma:minfeas} all points $(w^R,w^I)$ such that $w^R < 0.77\bm w$ are infeasible, $(\hat w^R,\hat w^I)$ is infeasible.
\end{proof}
\qed

\paragraph{$q_{21}$ lower bound}

Consider constraint (\ref{qbound}). Let $g_3 = (\bm{q}_2^g)^l - \bm q_2^d +\frac{\bm b}{\bm w}((w^R)^2+(w^I)^2) - \bm bw^R - \bm gw^I$.

\begin{lemma}\label{lemma:qbnd}
	KKT points of problem (\ref{nlpi}), $i = 3$ do not violate the boundary-invexity for Model \ref{model:ac_opf_1}.
\end{lemma}
\begin{proof}
(\ref{nlpi}) takes the following form for $i=3$:

\begin{align*}
& \max ~\bm c_1(\bm g\bm w-\bm gw^R-\bm bw^I) + c_2(\frac{\bm g}{\bm w}((w^R)^2+(w^I)^2) - \bm gw^R + \bm bw^I) \\
& s.t. ~-\frac{\bm b}{\bm w}((w^R)^2 + (w^I)^2) = (\bm{q}_2^g)^l - \bm q_2^d-\bm bw^R-\bm gw^I
\end{align*}

which can be rewritten as

\begin{align*}
& \max ~\bm c_1(\bm gw-\bm gw^R-\bm bw^I) + \frac{\bm c_2}{\bm b}((\bm b^2+\bm g^2)w^I - \bm g((\bm{q}_2^g)^l - \bm q_2^d)) \\
& s.t. ~\frac{(w^R)^2 + (w^I)^2}{\bm w} = -\frac{(\bm{q}_2^g)^l - \bm q_2^d}{\bm b}+w^R+\frac{\bm g}{\bm b}w^I
\end{align*}

The KKT conditions for this problem are:

\begin{align*}
\bm c_1\bm g &= -\lambda(\frac{2\hat w^R}{\bm w} - 1) \\
-\bm c_1\bm b - \frac{\bm c_2(\bm b^2+\bm g^2)}{\bm b} &= -\lambda(\frac{2\hat w^I}{\bm w} - \frac{\bm g}{\bm b}) \\
~\frac{(\hat w^R)^2 + (\hat w^I)^2}{\bm w} &= -\frac{(\bm{q}_2^g)^l - \bm q_2^d}{\bm b}+\hat w^R+\frac{\bm g}{\bm b}\hat w^I
\end{align*}
and the first equation implies that
$$\hat w^R = -\frac{\bm c_1\bm g\bm w}{2\lambda} + \frac{\bm w}{2}.$$

The solution of this system can violate boundary-invexity only if $\lambda > 0$. Then $\hat w^R < \frac{\bm w}{2}$. But since by Lemma \ref{lemma:minfeas} all points $(w^R,w^I)$ such that $w^R < 0.77\bm w$ are infeasible, $(\hat w^R,\hat w^I)$ is infeasible.
\end{proof}
\qed

We now consider the thermal limit constraint (\ref{therm_lim2}). 

\begin{lemma}\label{lemma:therm}
	If constraint (\ref{therm_lim2}) is non-redundant in a given subset, it is locally convex in this subset.
\end{lemma}
\begin{proof}
Consider the boundary of the set defined by constraint (\ref{therm_lim2}). It is given by:

\begin{align*}
& (\frac{\bm g}{\bm w}((w^R)^2+(w^I)^2) - \bm gw^R + \bm bw^I)^2 + (-\frac{\bm b}{\bm w}((w^R)^2+(w^I)^2) + \bm bw^R + \bm gw^I)^2 \\
& - \bm s^u = \frac{\bm g^2}{\bm w^2}s^2 + \frac{2\bm g}{\bm w}s(\bm bw^I-\bm gw^R) + (\bm bw^I-\bm gw^R)^2 + \frac{\bm b^2}{\bm w^2}s^2 - \frac{2\bm b}{\bm w}s(\bm bw^R+\bm gw^I) \\
& + (\bm bw^R+\bm gw^I)^2 - \bm s^u = \frac{|\bm Y|}{\bm w^2}s^2 + \frac{2}{\bm w}s(-\bm g^2w^R-\bm b^2w^R) + (\bm bw^I-\bm gw^R)^2 \\
& + (\bm bw^R+\bm gw^I)^2 - \bm s^u = \frac{|\bm Y|}{\bm w^2}s^2 - \frac{2w^R|\bm Y|}{\bm w}s + |\bm Y|s - \bm s^u = 0
\end{align*}

where $s = (w^R)^2+(w^I)^2$ and $|\bm Y| = \bm g^2 + \bm b^2$. This equation has the following solutions:

\begin{align*}
%& s^2 + s(w^2 - 2xw) - \frac{tw^2}{Y} \leq 0 \\
& s = \frac{\bm w}{2}(2w^R - \bm w - \sqrt{(2w^R-\bm w)^2 + \frac{4\bm s^u}{|\bm Y|}}) \text{ and} \\
& s = \frac{\bm w}{2}(2w^R - \bm w + \sqrt{(2w^R-\bm w)^2 + \frac{4\bm s^u}{|\bm Y|}})
\end{align*}

The first equation has no solution since $s$ is non-negative and the right-hand side is negative. Now we can write the thermal limit constraint as:

$$(w^I)^2 \leq \frac{\bm w}{2}(2w^R - \bm w + \sqrt{(2w^R-\bm w)^2 + \frac{4\bm s^u}{|\bm Y|}}) - (w^R)^2$$

Let $R = \sqrt{(2w^R-\bm w)^2 + \frac{4\bm s^u}{|\bm Y|}}$ and $\phi(w^R) = \frac{\bm w}{2}(2w^R - \bm w + R) - (w^R)^2$. Constraint (\ref{therm_lim2}) describes a convex set if $\phi(w^R)$ is concave. To obtain the conditions for its concavity, we will calculate the second derivative:

\begin{align*}
& \phi'(w^R) = \frac{\bm w}{2}(2+R') - 2w^R \\
& \phi''(w^R) = \frac{\bm w}{2}R'' - 2 = \frac{\bm w}{2}\frac{4R - \frac{4(2w^R-\bm w)^2}{R}}{R^2} - 2
\end{align*}

A function is concave if its second derivative is negative:

\begin{align*}
& \frac{\bm w}{2}\frac{4R - \frac{4(2w^R-\bm w)^2}{R}}{R^2} - 2 < 0 \\
& R^2 - (2w^R-\bm w)^2 < \frac{R^3}{\bm w}
\end{align*}

Observe that, from the definition of $R$, the left hand side of this inequality is equal to $\frac{4\bm s^u}{|\bm Y|}$:

\begin{align*}
& \frac{4\bm s^u}{|\bm Y|} < \frac{R^3}{\bm w} ~\Leftrightarrow~ (\frac{4\bm s^u\bm w}{|\bm Y|})^{\frac{2}{3}} < R^2  ~\Leftrightarrow~\\
& \sqrt[3]{(\frac{4\bm s^u\bm w}{|\bm Y|})^2} < (2w^R-\bm w)^2 + \frac{4\bm s^u}{|\bm Y|}  ~\Leftrightarrow~\\
& w^R > \frac{1}{2}\sqrt{\sqrt[3]{(\frac{4\bm s^u\bm w}{|\bm Y|})^2} - \frac{4\bm s^u}{|\bm Y|}}+ \frac{\bm w}{2}
\end{align*}

Let $\psi(x) = x^{\frac{2}{3}}\bm w^{\frac{2}{3}} - x$. Find the stationary point of $\psi(x)$:

\begin{align*}
& \psi'(x) = \bm w^{\frac{2}{3}}\frac{2}{3}\frac{1}{\sqrt[3]{x}} - 1 = 0 \\
& x = \frac{8\bm w^2}{27}
\end{align*}

To verify the second order optimality condition, calculate the second derivative:

$$\psi''(x) = -\frac{2}{9}\bm w^{\frac{2}{3}}\frac{1}{\sqrt{x^4}} < 0$$

Hence $\psi(x)$ is concave and

$$\psi(\frac{8\bm w^2}{27}) = (\frac{8\bm w^2}{27})^{\frac{2}{3}}\bm w^{\frac{2}{3}} - \frac{8\bm w^2}{27} = \frac{4\bm w^2}{27}.$$

We have shown that $\psi(x) \leq \frac{4\bm w^2}{27} ~\forall x > 0$. Then we can guarantee that constraint (\ref{therm_lim2}) is convex if 
$$w^R > \bm w(\frac{1}{3\sqrt{3}} + 0.5).$$ 
Since $(\frac{1}{3\sqrt{3}} + 0.5) < 0.77$ and, by Lemma \ref{lemma:minfeas}, all $(w^R,w^I)$ such that $w^R < 0.77\bm w$ are infeasible, (\ref{therm_lim2}) is convex everywhere where it is non-redundant.
\end{proof}
\qed

\begin{corollary}
	Model \ref{model:ac_opf_1} is boundary-invex.
\end{corollary}
\begin{proof}
	Based on Lemmas \ref{lemma:wbnd}-\ref{lemma:therm}, we can show that all KKT points for the auxiliary \ref{nlpi} problems are infeasible with respect to Model \ref{model:ac_opf_1}.
	Based on Definition \ref{def:bi}, boundary-invexity is established.
\end{proof}
\qed

\section{Conclusion}\label{sec:conclusion}
Given a non-convex optimization problem, boundary-invexity captures the behavior of the objective function on the boundary of its feasible region. 
In this work, we show that boundary-invexity is a necessary condition for KT-invexity, that becomes sufficient in the two-dimensional case. 
Unlike conventional invexity conditions, boundary-invexity can be verified algorithmically and in some cases in polynomial-time.
This is a first step in extending the reach of interior-point methods to non-convex problems. Future research directions include extending the sufficiency proof to the n-dimensional case and deriving conditions for checking the connectivity of non-convex sets.

	%	\bibliographystyle{spmpsci}      % mathematics and physical sciences
	\bibliographystyle{siamplain}
	\bibliography{invexity}   % name your BibTeX data base

\end{document}